\documentclass[12pt,makeidx]{amsart}
\usepackage{amssymb,mathrsfs, amsmath, amsthm}
\usepackage{url,titletoc,enumitem}

\usepackage[usenames,dvipsnames]{xcolor}
\definecolor{darkblue}{rgb}{0.0, 0.0, 0.55}
\usepackage[pagebackref,colorlinks,linkcolor=BrickRed,citecolor=OliveGreen,urlcolor=darkblue,hypertexnames=true]{hyperref}

\renewcommand{\qedsymbol}{\rule[.12ex]{1.2ex}{1.2ex}}
\renewcommand{\subset}{\subseteq}
\renewcommand{\emptyset}{\varnothing}

\DeclareMathOperator{\tr}{tr}

\newcommand{\simu}{\stackrel u{\sim}}

\makeatletter
\newcommand{\mycontentsbox}{%
{\centerline{NOT FOR PUBLICATION}
\addtolength{\parskip}{-2.3pt}
\small\tableofcontents}}
\def\enddoc@text{\ifx\@empty\@translators \else\@settranslators\fi
\ifx\@empty\addresses \else\@setaddresses\fi
\newpage\mycontentsbox\newpage\printindex}
\makeatother

\newtheorem{theorem}{Theorem}[section]
\newtheorem{cor}[theorem]{Corollary}
\newtheorem{lemma}[theorem]{Lemma}

\newtheorem{prop}[theorem]{Proposition}

\newtheorem{remark}[theorem]{Remark}

\newtheorem{example}[theorem]{Example}
\newtheorem{thm}[theorem]{Theorem}
\newtheorem{lem}[theorem]{Lemma}

\newtheorem*{lemma*}{Lemma}

\def\beq{\begin{equation}}
\def\eeq{\end{equation}}

\def\ben{\begin{enumerate}}
\def\een{\end{enumerate}}

\newcommand{\pre}[1]{{#1}^{\rm re}}
\newcommand{\pim}[1]{{#1}^{\rm im}}

\numberwithin{equation}{section}

\def\cD{ {{\mathcal D}}}

\def\cN{ {{\mathcal N}}}
\def\cP{ {{\mathcal P}}}

\def\cX{ {{\mathcal X}}}
\def\cEE{ {{\mathcal E}}}

 \def\bbF{ {\mathbb C}}

\def\R{ {\mathbb{R}} }
\def\C{ {\mathbb{C}} }
\def\F{\mathbb{C}}

\def\N{ {\mathbb{N}} }

\def\cA{ {\mathcal A} }
\def\cB{ {\mathcal B} }
\def\cC{ {\mathcal C} }
\def\cD{ {\mathcal D} }

\def\cL{ {\mathcal L} }

\def\cN{ {\mathcal N} }

\def\cP{{\mathcal P}}

\def\cS{{\mathcal S} }
\def\cT{{\mathcal T}}

\def\scriptsize{}

 \def\la{{\lambda}}

\def\mfT{\mathfrak{T}}

\def\cS{{\mathcal S}}

\newcommand{\df}[1]{{\bf{#1}}{\index{#1}}}

\def\La{\Lambda}
\def\ot{{\otimes}}
\def\Cr{C_1}
\def\Cl{C_2}

\def\oX{X^{\mbox{\rm{b}}}}

\def\be{\mathbf{e}}

\def\hh{T}
\def\hhs{\mathfrak{T}}

\def\size{d}
\def\Fdd{M_d(\mathbb C)}
\def\cs{\stackrel{\rm c.s.}{\sim}}

\makeindex

\title{Circular Free Spectrahedra}

\author[E. Evert]{Eric Evert}
\address{Eric Evert, Department of Mathematics\\
  University of California \\
  San Diego
   }
   \email{eevert@ucsd.edu}

\author[J.W. Helton]{J. William Helton${}^1$}
\address{J. William Helton, Department of Mathematics\\
  University of California \\
  San Diego}
\email{helton@math.ucsd.edu}
\thanks{${}^1$Research supported by the NSF grant
DMS 1201498, and the Ford Motor Co.}

\author[I. Klep]{Igor Klep${}^{2}$}
\address{Igor Klep, Department of Mathematics, 
The University of Auckland, New Zealand}
\email{igor.klep@auckland.ac.nz}
\thanks{${}^2$Supported by the Marsden Fund Council of the Royal Society of New Zealand. Partially supported by the Slovenian Research Agency grants P1-0222 and L1-6722. }

\author[S. McCullough]{Scott McCullough${}^3$}
\address{Scott McCullough, Department of Mathematics\\
  University of Florida\\ Gainesville 
   }
   \email{sam@math.ufl.edu}
\thanks{${}^3$Research supported by the NSF grant DMS-1361501}



\begin{document}

\setcounter{tocdepth}{3}
\contentsmargin{2.55em} 
\dottedcontents{section}[3.8em]{}{2.3em}{.4pc} 
\dottedcontents{subsection}[6.1em]{}{3.2em}{.4pc}
\dottedcontents{subsubsection}[8.4em]{}{4.1em}{.4pc}

\begin{abstract}
This paper considers matrix convex sets   invariant under several types of rotations. 
It is known that matrix convex sets that are free semialgebraic 
are solution sets of Linear Matrix Inequalities (LMIs); they are called free spectrahedra.
We  classify all free spectrahedra that are circular, that is, closed under multiplication  by $e^{i t}$: 
up to unitary equivalence, the coefficients of a minimal LMI defining a circular free spectrahedron have a common block decomposition in which the only nonzero blocks are on the superdiagonal.\looseness=-1

A matrix convex set is called free circular if it is closed under left multiplication by unitary matrices. 
As a consequence of 
a Hahn-Banach  separation theorem for free circular matrix convex sets,
we show  the coefficients of a minimal LMI defining a free circular free spectrahedron  have, up to unitary equivalence, a block decomposition as above  with only two
blocks. 

This paper also gives a classification of those  noncommutative polynomials invariant
 under conjugating each coordinate  by a different unitary matrix. 
Up to unitary equivalence 
such a polynomial must be a direct sum of univariate polynomials.
\end{abstract}

\maketitle

\section{Introduction}
 For square matrices $A,B $,  write
 $A\preceq B$ (resp. $A \prec B$) to express that $B-A$
  is positive semidefinite (resp. positive definite). 
Given a $g$-tuple $A=(A_1, \dots A_g) \in M_d(\C)^g$, let $\Lambda_A(x)$ denote the linear matrix polynomial
\beq\label{eq:homPencil}
 \Lambda_A(x) =\sum_{j=1}^g A_j x_j
\eeq
 and let $L_A$ denote  the \df{(symmetric monic) linear pencil} 
\beq\label{eq:monPencil}
L_A (x)=I_d- \sum_{j=1}^g A_j x_j - \sum_{j=1}^g A_j^* x_j^* =I_d -\Lambda_A(x)-\Lambda_A(x)^* .
\eeq
The \df{spectrahedron} $\mathscr S_{A}$ is the set of all
 $x \in \C^g$ satisfying 
the \df{linear matrix inequality} (\df{LMI})
 $L_A (x) \succeq 0$. 
 Spectrahedra and LMIs are ubiquitous 
 in control theory \cite{SIG97,BGFB94}
 and optimization \cite{BPR13}.
 Indeed LMIs are at the heart of the subject called semidefinite programming.
 
 This article investigates spectrahedra from the perspective of the emerging areas of
free convexity
\cite{DDSS,Eff09,EW,Far12,HKM16,WW99,Wit84,Zal}
 and free analysis
 \cite{AM14,BVM,HKM12,KVV14,KS,Pop08,Tay72,Voi10}.
In free analysis we are interested in matrix variables 
and evaluate a linear pencil on 
$g$-tuples 
 $X=(X_1, \dots, X_g) \in M_n(\C)^g$ according to the formula
\beq\label{eq:evalPencil}
L(X)=I_d \ot I_n - \sum_{j=1}^g A_j \otimes X_j 
- \sum_{j=1}^g A_j^* \otimes X_j^*.
\eeq
For positive integers $n$, let
\beq\label{eq:LMIdomain}
\cD_{A}(n)=\big\{X\in M_n(\C)^g : L_A(X)\succeq0 \big\}.
\eeq
The sequence  $\cD_A = (\cD_A(n))_n$ is called  a \df{free spectrahedron}. It is the set of all solutions to the 
ampliated
LMI corresponding to
$L_A$. In particular, $\cD_A(1)=\mathscr S_A$.
Free spectrahedra are closely connected with operator systems for which
 \cite{FP12,KPTT,Arv08} are a few recent references.  
In a different direction they provide a model for convexity phenomena
in  linear system engineering problems described entirely by 
signal flow diagrams \cite{dOHMP09}.

 The main results of this article characterize 
free spectrahedra   and free polynomials  that are invariant under various natural types of circular symmetry.
  A core motivation for this article 
 comes from
 classical several complex variables where the  study of maps on various
 types of domains is a major theme.  There
 an important  class is the {\it circular 
 domains}. These behave very well under bianalytic mappings
 as described e.g.~by  Braun-Kaup-Upmeier \cite{BKU78}.\looseness=-1

\subsection{Main Results} 
\label{sec:mainpreview}
This subsection contains a  summary of the main results of the paper. 
  Let $M(\bbF)^g$ denote the sequence $(M_n(\bbF)^g)_{n\in\N}$ of $g$-tuples of $n\times n$ matrices with entries from $\bbF$. A subset $\Gamma \subset M(\bbF)^g$ is a sequence $(\Gamma(n))_n$ where $\Gamma(n) \subset M_n(\bbF)^g$. 

\subsubsection{Rotationally invariant free spectrahedra}
A subset $\cD \subset  M(\C)^g$ is \df{circular} 
 if $Z \in \cD$ implies $e^{i t} Z \in \cD$ for all $t\in\R$ and  is \df{free circular}  if $UZ \in \cD$ for each $n$, each $Z \in \cD(n)$, and each $n \times n$ unitary matrix $U \in M_n(\C)$. Here $U Z= ( U Z_1, \dots, U Z_g)$.
 Geometric and analytic properties of 
circular subsets of $\C^n$ and their generalizations, such as Reinhardt domains, are heavily investigated in several complex variables \cite{Kra01}, cf.~\cite{BKU78}.

Given a tuple $A\in M_d(\C)^g$, if there is an orthogonal decomposition of $\C^d$ such that with respect to this decomposition $A=A^1 \oplus A^2,$ then $L_A(x)=(L_{A^1} \oplus L_{A^2}) (x).$ In this case each  $L_{A^i}$ is a \df{subpencil} of $L_A$. If $\cD_A=\cD_{A^i},$ then $L_{A^i}$ is a \df{defining subpencil} for $\cD_A$. Say the pencil $L_A$ is a \df{minimal defining pencil} for $\cD_A$ if no proper subpencil of $L_A$ is a defining subpencil for $D_A$.

Theorem \ref{theorem:circularmain}  below
says
the  tuple $A$ in a minimal defining pencil $L_A$ of a circular free spectrahedron is
(up to unitary equivalence)
block superdiagonal. It also says, if the domain is free circular, then there are just two blocks. We refer to such a domain as a \df{matrix pencil ball}. 

\begin{theorem} 
\label{theorem:circularmain}
 Let  $A\in M_d(\C)^g$  and 
 suppose  $L_A$ is a minimal defining pencil for $\cD_A$.
 \ben[label={\rm(\arabic*)}]
 \item
 \label{it:circ} 
    Assume $A$ has no reducing subspace. 
   The free spectrahedron
  $\cD_{A}$ is circular if and only if 
  there is an orthogonal decomposition
  of $\mathbb C^d$ such that, with respect to this decomposition, the $A_s$ have the block decomposition
\beq
\label{eq:1offdiag}
  A_s =\begin{pmatrix} 0 & A_s(1) & 0 &  \cdots  & 0 \\
       0 & 0 & A_s(2) & \ddots & 0\\
       \vdots & \vdots & \ddots & \ddots &\vdots \\
       0 & 0 & \ddots &\ddots & A_s(k) \\ 
       0 & 0 & 0 & \cdots & 0 \end{pmatrix}, 
\eeq
  where the $A_s(j)$ are matrices of appropriate sizes and for each $j$ there exists at least one $s_j$ such that $A_{s_j} (j) \neq 0$.

   In any case,  $\cD_A$ is circular if and only if  the $A_s$ can be written as a direct sum of 
block superdiagonal
   matrices of the form \eqref{eq:1offdiag}.
   
   \item 
   \label{it:freecirc} 
The free spectrahedron  $\cD_A$ is  free circular
if and only if
 there exist $s,t\in\N$ with $s+t=d$ and a tuple $F$ of $s\times t$ matrices  such that 
$A$ is unitarily equivalent to 
\beq\label{eq:pball}
E= \begin{pmatrix} 0 & F \\ 0 & 0\end{pmatrix}.
\eeq
 \een
\end{theorem}

\begin{proof}
Part \ref{it:circ} is proved in  Section \ref{sec:circLMI} 
by a  geometric argument.  
    In strong contrast, the proof of Part \ref{it:freecirc} -- given in  Section \ref{sec:FreeCircSpec}, see Theorem \ref{thm:EWball} and Corollary \ref{cor:mainball} --
depends on
a strengthening (Proposition \ref{prop:sharp})
 of the characterization \cite[Proposition 3.5]{BVM} of free circular matrix convex sets (i.e., a
  version of the Effros-Winkler Theorem \cite{EW} for free circular  matrix convex sets). We give a self-contained proof
  of the latter  in Appendix \ref{sec:FreeCircDomain},
  see Theorem \ref{thm:OperatorBall}.\looseness=-1
\end{proof}

\subsubsection{Rotationally invariant free polynomials}

A free $d \times d$ matrix polynomial
$p$ is \df{invariant under coordinate unitary conjugation} 
if for any $n$,  and
 any  $g$-tuple of unitaries $U=(U_1, \dots, U_g) \in M_n(\C)^g$ 
 there exists a unitary $W$ 
such that for all $X \in M_n(\C)^g$,
\beq
\label{eq:CoordInvarDef}
p(U^*_1 X_1 U_1, \dots , U^*_g X_g U_g)= W^*  p(X) W.
\eeq
Our main theorem on polynomials characterizes monic free matrix polynomials
that are invariant under coordinate unitary conjugation. 

\begin{theorem} 
\label{theorem:invarpolyintro}
If $p$ is a monic free matrix polynomial, 
 then $p$ is invariant under coordinate unitary conjugation if and only if
\[
p(x)\simu 
p_1(x_1) \oplus \cdots \oplus p_g(x_g).
\]
That is, $p$ must be (up to unitary equivalence) a direct sum of univariate matrix polynomials. 
\end{theorem}

\begin{proof}
 The proof appears in Section  \ref{sec:UnitaryConjPoly}.
\end{proof}

\subsection{Readers Guide}
\label{sec:guide}
 In Section \ref{sec:circLMI} we prove Theorem \ref{theorem:circularmain}~\ref{it:circ} -- the  classification of all circular free spectrahedra; i.e., free spectrahedra that are closed under rotations by $e^{i t}.$   
In Section \ref{sec:FreeCircSpec}  we characterize
free circular spectrahedra 
thus finishing the proof of
 Theorem \ref{theorem:circularmain}.
    Finally, in Section \ref{sec:UnitaryConjPoly} we turn our attention to free matrix polynomials and
 prove Theorem \ref{theorem:invarpolyintro}.
Appendix \ref{sec:FreeCircDomain}  contains a self-contained
proof of the Ball-Marx-Vinnikov  Theorem \cite{BVM} classifying
 free circular
 matrix convex  sets. We prove a sharpened
 version by establishing an effective 
Hahn-Banach 
  separation result for free circular domains, see Proposition \ref{prop:2sharp}.

\section{Circular Free Spectrahedra}\label{sec:circLMI}
 A subset $\cD \subset  M(\C)^g$ is \df{circular} if $Z \in \cD$ implies $e^{i t} Z \in \cD$ for all $t\in\R$.
 In this section the first part of Theorem \ref{theorem:circularmain} characterizing circular free spectrahedra is established. The main idea of the proof is as follows. 
 Assuming $\cD_A$ is a circular free spectrahedron, 
 for  each $t\in\R$, the pencil $L_{e^{it}A}$ determines the same free spectrahedron
  as  $L_A$, namely $\cD_{A}=\cD_{{e^{it}A}}$.  
We are thus in a position to apply the Gleichstellensatz
(see e.g.~\cite[Theorem 1.2]{HKM13} and \cite[Theorem 1.2]{Zal})
characterizing when two free spectrahedra are the same.

\begin{remark}\rm
 \label{rem:redundant}
   It turns out if, for a $t$ such that $\frac{t}{\pi}$ is irrational,  $e^{it} A$ is  unitarily equivalent to $A$, then $\cD_A$ is circular. This fact is a corollary of the proof of Theorem \ref{theorem:circularmain}~\ref{it:circ} given below. For a direct proof, observe, if $e^{it} A = U^* A U$, then $e^{int} A = U^{*n} A U^n$ and thus, for a dense set of $t \in \mathbb R$, the tuple $e^{it} A$ is unitarily equivalent to $A$. A routine limiting argument completes the proof. 
\end{remark}

\subsection{Set up for the Proof of Theorem \ref{theorem:circularmain}~\ref{it:circ}}

Suppose $A$ satisfies the hypotheses of the theorem, except for possibly the irreducibility condition. Here, a $g$-tuple $A \in M_d(\C)^g$ is said to be \df{irreducible} if the $A_s$ have no common reducing subspace, i.e, if there is no proper subspace $M \subseteq M_d(\C)$ such that $A_s M \subseteq M$ and $A_s^* M\subseteq M$  for each $1\le s\le g$. 

We first present  the linear Gleichstellensatz 
adapted to our set up  of  free (non-symmetric) variables.

\begin{prop}\label{prop:Gsatz}
If $B\in M_e(\C)^g$ satisfies $\cD_A=\cD_B$, where
 $A\in M_d(\C)^g$ is minimal defining for $\cD_A$,
then
$B$ is unitarily equivalent to $A\oplus J$ for some $g$-tuple $J$. 
\end{prop}

\def\cDD{\mathscr D}
\begin{proof}
The statement holds when working over the field of real numbers and evaluating at tuples of symmetric matrices
by \cite[Theorem 1.2]{HKM13} and \cite[Theorem 1.2]{Zal}.
It is easy to see that the same proofs work over the field
of complex numbers and evaluating at tuples of self-adjoint matrices. We now reduce the proposition to this case.

To each monic pencil $L_A(x)$ in free variables $x,x^*$ we can
associate a monic pencil $\cL_{(\pre A,\pim A)}(y,z)$ with
self-adjoint coefficients in self-adjoint variables 
$y,z$ as follows. Let $\pre A_j=\frac12(A_j+A_j^*)$ 
and $\pim A_j=\frac1{2i}(A_j-A_j^*)$ for $j=1,\ldots,g$.
Then
\[
\cL_{(\pre A,\pim A)}(y,z)= I - \sum_{j=1}^g \pre A_j y_j -
\sum_{j=1}^g \pim A_j z_j.
\]
Each $X\in\cD_A$ yields a point $\frac12 \big(X+X^*,i(X-X^*)\big)$
in the free spectrahedron  (in self-adjoint variables)
$\cDD_{(\pre A,\pim A)}$. Conversely, given
$(Y,Z)\in\cDD_{(\pre A,\pim A)}$ we have
$Y-iZ\in\cD_A$. Hence $\cD_A=\cD_B$  implies that
$\cDD_{(\pre A,\pim A)}=\cDD_{(\pre B,\pim B)}$.

We claim that $\cL_{(\pre A,\pim A)}(y,z)$
is a minimal defining pencil for
$\cDD_{(\pre A,\pim A)}$. Indeed, as otherwise 
by the Gleichstellensatz (\cite[Theorem 1.2]{HKM13} or \cite[Theorem 1.2]{Zal}), there will be 
a
reducing subspace for $(\pre A,\pim A)$ and a
compression 
$\cL_{(\pre {\tilde A},\pim {\tilde A})}(y,z)$
of $\cL_{(\pre A,\pim A)}(y,z)$
to this subspace 
with
$\cDD_{(\pre A,\pim A)}=\cDD_{(\pre {\tilde A},\pim {\tilde A})}$. But this in turn will 
 yield a subpencil 
$L_{\tilde A}$ of $A$ with the same free spectrahedron
as $A$, contradicting the minimality of $L_A$.

Hence, again by the Gleichstellensatz, $(\pre A,\pim A)$
is (unitarily equivalent to) a subpencil of
$(\pre B,\pim B)$. But then $A$ is a subpencil of B, as desired.
\end{proof}

 Since, for each $t$, $A$ and $e^{it}A$ are minimal defining tuples for the free spectrahedron $\cD_A=\cD_{e^{it}A}$, by Proposition \ref{prop:Gsatz}, for each $t\in\R$ there is a unitary $U=U_t \in M_d(\C)$ such that,
 for each $s= 1, \dots, g,$
  \beq
  \label{eq:UAUe}
  U_t^*A_s U_t =e^{it}A_s.
  \eeq
 For a fixed $s$, equation \eqref{eq:UAUe} holds for each real $t$ so the spectrum of $A_s$ is a circular set for each $s$. Since each $A_s$
  is finite dimensional, the spectrum of each $A_s$ is $\{0\}$ 
  and each $A_s$ is nilpotent. 

 Fix a number $t$ relatively irrational with respect to $\pi$.
 For notational ease, abbreviate $U=U_t$ (for this $t$). Being unitary,
 the matrix $U$ can be (block) diagonalized as
\[
U=W^*DW
\]
where $D \in M_d(\C)$ is diagonal and $W \in M_d(\C)$ is unitary. Equation \eqref{eq:UAUe} shows
\[
D^*WA_sW^*D=e^{it}WA_sW^*
\]
Clearly, $L_{WAW^*}$ and $L_{A}$ define the same free spectrahedron. Thus, without loss of generality,  $U$ may be taken to have the form
\beq
\label{eq:ublock}
U = (\lambda_1 I_{m_1}\oplus \lambda_2 I_{m_2} \oplus \cdots \oplus
\lambda_{k+1} I_{m_{k+1}}),
\eeq
 where the $\lambda_j$ are distinct unimodular 
  numbers. Let $\mathcal{S}_j$ denote the corresponding eigenspace of $U$
   and let the $I_{m_j}$ be identity matrices on these spaces.

 Since $\C^d=\cS_1 \oplus \dots \oplus \cS_{k+1}$ we can use this orthogonal sum to give a block decomposition 
 \beq
\label{eq:Asblock}
A_s=( A_s(j,\ell) )_{j,\ell} )
\eeq
 subordinate to the $\cS_i$. 
 Note that 
 \beq
  \overline{\la_j} \la_\ell A_s(j,\ell)  = e^{i t} A_s(j,\ell) \qquad
 \eeq
so it follows that
 \beq
 \label{eq:dicot}
  \la_\ell  = e^{i t}   \la_j \qquad  \text{or}  
  \qquad A_s(j,\ell)=0 \ \ \ \text{for all } s.
 \eeq
 Equation \eqref{eq:dicot} implies  $A_s(j,j)=0$ for each $s$ and $j$.

\begin{lemma}
\label{lem:atmostone}
Let $U \in M_d(\C)$ be a unitary with the form of equation \eqref{eq:ublock} and let $A=(A_1, \dots, A_g) \in M_d(\C)^g$ 
be a $g$-tuple of matrices with block decomposition $A_s=(A_s (j,\ell)_{j,\ell})$ as described in equation \eqref{eq:Asblock}.
Assume there is a $t \in \R$ relatively irrational with respect to $\pi$ such that
$e^{it}A_s=U^*A_s U$
for all s.

 Given $1\le j,\hat{j}, \ell,\hat{\ell}\le k+1$, if $A_s (j,\ell) \neq 0$
 and if $A_{\hat{s}} (j, \hat{\ell}) \neq 0$, then, by equation \eqref{eq:dicot}, $\ell= \hat{\ell}$.
Likewise, if $A_s (j, \ell) \neq 0$ and if $A_{\hat{s}} (\hat{j}, \ell) \neq 0$, then $j=\hat{j}$. 
Moreover, if $(j, \ell)$ is a nonzero location, then, for $\hat{j} \neq j$ and $\hat{\ell} \neq \ell$
and all $s$, the matrices $A_s (\hat{j}, \ell)$ and $A_s (j, \hat{\ell})$ are both zero.
\end{lemma}

\begin{proof}
Fix $1 \leq j \leq k+1$ and note from equation \eqref{eq:dicot} that if 
$A_s(j, \ell)$ and $A_{\hat{s}} (j, \hat{\ell})$ are both not zero, then
$\lambda_\ell=e^{it} \lambda_j$ and $\lambda_{\hat{\ell}}=e^{it} \lambda_j$. 
In particular,  $\lambda_\ell= \lambda_{\hat{\ell}}$. Since the $\lambda_k$ are distinct
it follows that $\ell=\hat{\ell}$. 
Similarly if $A_s (j, \ell)$ and $A_{\hat{s}} (\hat{j}, \ell)$ are both not zero, then equation $\eqref{eq:dicot}$
shows $\lambda_j= \lambda_{\hat{j}}$, hence $j=\hat{j}$.
\end{proof}

 Given a family of matrices $A=\{A_s\}_{s=1}^g$ with the block decomposition $A_s=(A_s(j,\ell))_{j,\ell})$,
a  sequence of pairs from the set $\{1,\dots, k+1\}$ 
 of the form
 \beq
 \label{eq:chain}
 \cC =  \{(j_0,j_1),(j_1,j_2),(j_2,j_3),\dots, (j_m,j_{m+1}) \}
 \eeq
 such that for each $1\le r\le m$ there is an $s$ such that  $A_s(j_r,j_{r+1})\not =0$ is an \df{admissible chain}.
 Call $j_0$ the left end of $\cC$ and denote by $\cS_\cC$ the subspace
 \beq
 \label{def:cScC}
 \cS_\cC= \cS_{j_0} \oplus \cS_{j_1} \oplus \cS_{j_2} \oplus \cdots \oplus  \cS_{j_{m+1}}.
 \eeq
The family $A$ has a \df{block zero column} if there is an $\ell$ such that 
$A_s(j,\ell)=0$ for all $s,j$.

\begin{lemma}
\label{lem:zerocolumn}
Assume the setup and hypotheses of Lemma \ref{lem:atmostone} with chain structure as described in equation \eqref{eq:chain}.
\begin{enumerate}[label={\rm(\arabic*)}]
\item \label{it:distinct} If $\cC$ is a chain as in \eqref{eq:chain}, then the $j_k$ are distinct.
\item \label{it:blockzero} The family $A$ has a block zero column. 
\end{enumerate}
\end{lemma}

\begin{proof}
 Suppose $\cC$ is a chain as in \eqref{eq:chain}, but the $j_k$ are not distinct. Since $A_s(j,j)=0$ for all $j$ and $s$, in this case we may assume that $m\ge 1$ and $j_{m+1}=j_0$ and $j_k\ne j_\ell$ for $0\le k,\ell\le m$. By reindexing if needed, we may assume that
\[
 \cC =\{(m,1),(1,2),(2,3), \dots, (m-1,m)\}
\]
 is an admissible chain. Summarizing, for each  $1 \leq j <m$ there exists
an $s_j$ such that $A_{s_j} (j,j+1) \neq 0$ and there exists an $s_m$ such that
  $A_{s_m}(m,1)\ne 0$.   
 Equation \eqref{eq:dicot} implies
  $\lambda_j=\lambda_1 e^{(j-1)it}$ for each $1 \leq j \leq p.$ Thus $\lambda_p$
  must be both $\lambda_1 e^{-it}$ and $\lambda_1 e^{(p-1)it}$. Hence
 $pt$ is a multiple of $2\pi$ contradicting the choice of  $t$ as relatively irrational
with respect to $\pi$ and the proof of item \ref{it:distinct} is complete.

Turning to item \ref{it:blockzero} and  arguing by contradiction, suppose for each $\ell$ there exists
  a $j_\ell$ and an $s_\ell$ so that $A_{s_\ell} (j_\ell, \ell)\ne 0$.
  In this case,  since, by Lemma \ref{lem:atmostone}, each column and row has exactly one nonzero entry and since all diagonal
 entries of $A_s$ are zero, there is an $m$ and distinct indices $j_0,j_1,\dots,j_{m}$ such that
\[
 \cC =\{(j_0,j_{m}), (j_{m},j_{m-1}), \dots, (j_2,j_1), (j_1,j_0)\}
\]
is an admissible chain. An application of item \ref{it:distinct} concludes the proof.
\end{proof}

The following lemma completes the set up for the proof of Theorem \ref{theorem:circularmain}~\ref{it:circ}.

\begin{lemma}
\label{lem:irred}
Assume the set up and hypotheses of Lemma \ref{lem:zerocolumn} and assume $\cC$ is a maximal 
chain whose left end $j_0$ is a block zero column of the $A_s$. Then the following hold
\mbox{}\par
\begin{enumerate}[label={\rm(\arabic*)}]
\item
\label{it:invarsubspace}
 $\cS_\cC$ (defined in equation \eqref{def:cScC}) is a common reducing subspace for each $A_s$.
\item 
\label{it:reducingdecomp}
The restriction of each  $A_s$ to $\cS_{\cC} $ has the form of equation \eqref{eq:1offdiag} with respect to the orthogonal decomposition of $\cS$ as described by equation \eqref{def:cScC}.
\item
\label{it:irred}
 If $A$ is an irreducible family and  $A_s(\ell,\hat{j})=0$ for all $1\le \ell \le k+1$ and $1\le s\le g,$ then $\hat{j}=j_0$. In particular, the $A_s$ have exactly one block zero column. By reindexing if needed, $\{1,\dots,k+1\}$ is an admissible chain and $A_s(\ell,1)=0$ for all $s,j$. 
\end{enumerate}
\end{lemma}

\begin{proof}
Use the notations of equations \eqref{eq:chain} and \eqref{def:cScC}.  In particular, $A_s(j,k)$ maps $\cS_{k}$ into $\cS_{j}$.\looseness=-1

By the definition of chain, for each $1\le \ell \le m$, there is an $s_\ell$ such that $A_{s_\ell}(j_\ell, j_{\ell+1})\ne 0$. From Lemma \ref{lem:atmostone}, for each $1\le \ell \le m$, each  $j\ne j_{\ell+1}$ and each $1\le s\le g$ the matrix $A_s(j_\ell,j)=0$. Hence, $A_s  \cS_{j_{\ell+1}} \subset  \cS_{j_\ell}.$ On the other hand, $A_s(j,j_0)\cS_{j}=0$  by the choice of $j_0$.   It follows that $A_s \cS_{\cC} \subset \cS_{\cC}.$   Thus $\cS_{\cC}$ is a common {\it invariant} subspace for the $A_s$.    On the other hand, since, for each $0\le \ell \le m$ the location $(j_\ell,j_{\ell+1})$  is a nonzero location, Lemma \ref{lem:atmostone} shows $A_s(j_\ell,q)=0$ for all $q\notin \{j_0,\dots, j_{m+1} \}$ and all $1\le s\le g$. Finally, the existence of  a $q\notin \{j_0,\dots, j_{m+1} \}$ and an $s$ such that $A_s(j_{m+1},q)\ne 0$ contradicts the maximality of the chain $\cC$. Hence $A_s(j_{ml+1},q)=0$ for all such $q$ and all $s$.  It follows that $\cS^\perp$ is also a common invariant subspace for the $A_s$. Hence $\cS$ is reducing. 

Items (2) and (3) follow immediately from item (1) and the definition of irreducible.
\end{proof}

\subsection{Proof of Theorem \ref{theorem:circularmain}~\ref{it:circ}}

The initial set up of the proof shows that, up to unitary equivalence, the $A_s$
 are nilpotent matrices and that there exists a $t$ relatively irrational with respect
to $\pi$ and a unitary $U$ with the form of equation \eqref{eq:ublock} such that
\[
U^*A_s U=e^{it} A_s \quad\text{for all $s$.}
\]

Relative to the block decomposition for $U$, write $A_s=(A_s (j,\ell)_{j, \ell})$ 
(as described in equation \eqref{eq:Asblock}). Applying Lemma \ref{lem:atmostone} shows that if $A_s (j,\ell) \neq 0$
 and $A_{\hat{s}} (j, \hat{\ell}) \neq 0,$ then $\ell= \hat{\ell}$ and if $A_s (j, \ell) \neq 0$ and $A_{\hat{s}} (\hat{j}, \ell) \neq 0,$ then $j=\hat{j}$.

Applying Lemma \ref{lem:zerocolumn} shows that there is some $j_0$ such that $A_s (\ell, j_0)=0$ for all $s$ and $\ell$. 
It follows that the $A_s$ have a maximal admissible chain $\cC$ of the form 
\[
\cC =  \{(j_0,j_1),(j_1,j_2),(j_2,j_3),\dots, (j_m,j_{m+1}) \}
\]
whose left end $j_0$ is a block zero column of the $A_s$.

Applying Lemma \ref{lem:irred}~\ref{it:invarsubspace} shows that $\cS_{\cC}$ (as defined in
 equation \eqref{def:cScC}) is a common reducing subspace for the $A_s$. Lemma \ref{lem:irred}
\ref{it:reducingdecomp} and \ref{it:irred} show that the $A_s$ have the form of equation 
\eqref{eq:1offdiag} and complete the proof.
\hfill\qedsymbol

\subsection{Examples}
Here are two classical examples of circular free spectrahedra.

\begin{example}\rm
{The \bf Bi-disk} is a circular free spectrahedron  given as the positivity set of 
\beq
 L_A(z)
= 
\begin{pmatrix}
  1& z_1     \\
  z_1^* & 1  
  \end{pmatrix}
  \oplus
  \begin{pmatrix}
  1& z_2     \\
  z_2^* & 1  
  \end{pmatrix}
\eeq

\end{example}

\begin{example}\rm
The {\bf Ball}  is a circular free spectrahedron  given as the positivity set of 
\beq
L_A(z)=
\begin{pmatrix}
1 & z_1 & z_2 & \cdots & z_g \\
z_1^* & 1 & 0 & \cdots & 0 \\
z_2^* & 0 & 1 & \cdots & 0 \\
\vdots & \vdots & \vdots & \ddots & \vdots \\
z_g^* & 0 & 0 & \cdots  & 1
\end{pmatrix}
\eeq

\end{example}

\section{A Free Circular Free Spectrahedron is a Matrix Pencil Ball}\label{sec:FreeCircSpec}
This section contains  the proof of Theorem \ref{theorem:circularmain}~\ref{it:freecirc}.
 Throughout, $A\in M_d(\bbF)^g$ is a fixed tuple of $d\times d$ matrices and it is assumed that the free spectrahedron 
 $\cD_A$
 is free circular. We will state precisely and prove in Theorem \ref{thm:EWball} below there is an $N$ (at most $d^3$)
 and a tuple $F\in M_N(\bbF)^g$ such that,  $\cD_A=\cD_E$, where
 \[
  E = \begin{pmatrix} 0 & F \\ 0 & 0 \end{pmatrix}.
 \]
 A separate argument, given as Corollary \ref{cor:mainball}, shows in fact, if $A$ is minimal, then $A$ is unitarily equivalent to $E$. Thus, in any case $E$ can be chosen to be of size $d$. 

\subsection{Free Circular Matrix Convex Sets}\label{sec:mconvex}

In this section we describe free circular matrix convex sets. 
 The set $\Gamma\subset M(\C)^g$ is \df{matrix convex} \cite{EW} if it is \df{closed under direct sums} in the sense that if $X\in \Gamma(n)$ and $Y\in\Gamma(m)$, then the tuple $X\oplus Y$ whose $j$-th entry is\looseness=-1
\[
 X_j\oplus Y_j =\begin{pmatrix} X_j&0\\0& Y_j \end{pmatrix}
\]
is in $\Gamma(n+m)$; and is \df{closed under isometric conjugation} in the sense that  if $X\in\Gamma(n)$ and $V$ is an $n\times m$ isometric matrix, then
\[
 V^*XV = \begin{pmatrix} V^* X_1 V, & \ldots ,& V^* X_g V \end{pmatrix} \in \Gamma(m).
\]
 In the case $0\in\Gamma(1)$, if $\Gamma$ is closed under direct sums and isometric conjugation, then it is closed under contractive conjugation (replacing $V$ isometric with $V$ contractive) \cite{HM04}.  It is not hard to show, if $\Gamma$ is matrix convex, then each $\Gamma(n)$ is convex in the conventional sense.

The  Effros-Winkler  matricial Hahn-Banach separation theorem \cite{EW} says if $\Gamma$ is closed (meaning each $\Gamma(n)$ is closed), matrix convex,  $0\in \Gamma(1)$,  and if $Y\notin M_n(\bbF)^g \setminus \Gamma(n)$, then there exists a tuple $A\in M_n(\bbF)^g$ such that  $L_A(X)\succeq 0$ for $X\in\Gamma$, but $L_A(Y)\not\succeq 0$. In this sense $\cD_A$ is the free analog of a separating hyperplane and a  closed matrix convex set is an intersection of free spectrahedra. 

Proposition \ref{prop:sharp} below is the analog of the Effros-Winkler separation theorem
for free circular matrix convex sets.
It is an effective version of \cite[Proposition 3.5]{BVM}.

\begin{lemma}
\label{lem:VXV}
 Suppose $\cD \subset  M(\bbF)^g$ contains $0$ and is closed with respect to direct sums. If for each pair of positive integers $s,t$, each $Y\in\cD(t)$ and each pair of $t\times s$  isometries $V_1,V_2$ (so $t\ge s$),  $V_2^* X V_1 \in \cD(s)$, then for each pair $m,n$ of positive integers, each $X\in\cD(n)$ and each pair $C_1,C_2$ of $m\times n$ contractions,   $\Cl^* X \Cr\in \cD(m)$.
\end{lemma}

\begin{proof}
Let positive integers $m,n$, a tuple $X\in\cD(n)$ and a pair of $m\times n$ contractions $C_1,C_2$ be given.
 Let $D_j=(I-C_j^*C_j)^{\frac 12}$.  With this choice of $D_j$, the $(m+n)\times m$ matrices
\[
 V_j =\begin{pmatrix} C_j & D_j \end{pmatrix}
\]
 are isometries.  Since $\cD$ is closed with respect to direct sums and contains $0,$ it follows that  $X\oplus 0\in \cD(n+m).$ Since $\cD$ is closed with respect to multiplying on the left by the adjoint of an isometry and the right by an isometry (of the same sizes), 
\[
  V_2^* \begin{pmatrix} X & 0 \\ 0 & 0 \end{pmatrix} V_1 =  \Cl^* X \Cr \in \cD(m). \qedhere
\]
\end{proof}

Following \cite{BVM} we call 
a graded set $\cC=(\cC(n))_{n\in\N}$  \df{matrix balanced} 
if for each pair $m,n$ of positive integers, each  $X\in \cC(n)$ and  pair of $n\times m$ contractions $C_1,C_2$, the matrix $\Cl^* X \Cr \in \cC(m)$. 
Observe, if $\cC$ is matrix balanced and closed with respect to direct sums, then it is matrix convex and in particular each $\cD_A(n)$ is convex 
in the ordinary sense.

 \begin{prop}
 \label{prop:convexfreecircular}
   A subset  $\cD$ of $M(\F)^g$  is closed with respect to direct sums and  matrix balanced if and only if it is matrix convex,  free circular and contains $0$.
\end{prop}

\begin{proof}
 Choosing $\Cl=Z^*$  and $\Cr=I$  shows if $\cD$ is matrix balanced, then it is free circular.  Choosing $\Cr=\Cl$ shows matrix balanced implies matrix convex. Choosing either $\Cr$ or $\Cl$ equal zero shows $0\in\cD$.  Hence, $\cD$ matrix balanced implies matrix convex, free circular and $0\in \cD$. 

In view of Lemma \ref{lem:VXV}, it suffices to prove the converse under the added assumption that the $C_j$ are isometries. In this case, there exists an $n\times n$ unitary matrix $W$ such that $W\Cr=\Cl$.  Letting $Z=W^*$ gives $Z X \in \cD(n)$  by the free circular hypothesis. Thus $\Cr^* (ZX)\Cr\in \cD(m)$ by the matrix convex assumption. Finally, as $\Cr^* Z= \Cl^*$ the result follows.
\end{proof}

  Given $\epsilon>0$,
  the {\bf free $\epsilon$-neighborhood of $0$},
  denoted  $\cN_\epsilon,$ is the graded set 
  $(\cN_\epsilon(n))_{n=1}^\infty$ where
 \[
   \cN_\epsilon(n) =\{X\in M_n(\bbF)^g : \sum \| X_j \| < \epsilon  \}. 
 \]

\begin{prop}
 \label{prop:sharp}
 Let $\cC=(\cC(n))$ denote a  
free circular 
 matrix convex
 subset of the graded set  $M(\bbF)^g$ that 
 contains 
 a free $\epsilon$-neighborhood of $0$.
 If $\oX\in M_n(\bbF)^g$ is in the boundary of $\cC(n)$,
 then there is a tuple $Q\in M_n(\bbF)^g$ such that 
 $\|\Lambda_Q(Y)\|\le 1$ for all $m$ and
 $Y\in\cC(m)$  and such that  $\|\Lambda_Q(\oX)\|=1$. 
\end{prop}

\begin{proof}
By Proposition \ref{prop:convexfreecircular} and  \cite[Proposition 3.5]{BVM},
 $\cC = \mathcal B_F$ for some operator tuple $F$ acting on a Hilbert space $H$.   Here $\mathcal B_F$ is the operator pencil ball determined by $F$, i.e.,
\[
\mathcal B_F = \Big\{X\in M(\C)^g : \big\|\sum_j F_j\otimes X_j\big\|\leq1\Big\}.
\]
Let $\Lambda_F(x)=\sum_{j=1}^g F_j x_j$ denote the homogeneous operator pencil determined by $F.$
 Since $\oX$ is in the boundary of $\mathcal B_F$, we see that $\|\Lambda_F(\oX)\|=1$. Hence, there exists a sequence of unit vectors $\gamma_k \in H \otimes\C^n$ such that $(\|\Lambda_F(\oX)\gamma_k\|)_k$ tends to $1$. Fix $k$. Write $\gamma_k = \sum_{j=1}^n \gamma_{k,j}\otimes e_j$. Let $\Gamma_k$ denote an $n$ dimensional subspace of $H$ containing the span of $\{\gamma_{k,1},\dots,\gamma_{k,n}\}$ (if the dimension of $H$ is less than $n$, then there is nothing to prove) and let $G^k= V^* FV\in M_n(\C)^g$, where $V:\C^n\to \Gamma_k$ is   an isometry. It follows that $(\|\Lambda_{G^k}(\oX)\|)_k$ tends to $1$. By compactness, $(G^k)$ has a subsequence which converges in norm to some $G\in M_n(\C)^g$. It follows that $\|\Lambda_{G}(\oX)\|=1$ and $\Lambda_{G}$ is at most one in norm on $\cC$.
\end{proof}

The authors of \cite{BVM} obtain \cite[Proposition 3.5]{BVM}  as a consequence of Ruan's 
representation theorem for operator spaces (see \cite[Theorem 2.3.5]{ER} or \cite[Chapter 13]{Pau}). 
We give an elementary self-contained proof of
Proposition \ref{prop:sharp} in Appendix \ref{sec:FreeCircDomain}.

\subsection{Criteria for Membership in a Free Spectrahedron}
This section contains three simple lemmas preliminary to the proof
 of Theorem \ref{theorem:circularmain}~\ref{it:freecirc}.

\begin{lemma}
 \label{lem:incDA}
 A tuple $X\in M_n(\bbF)^g$ lies in $\cD_A(n)$ if and only if for every subspace $M$ of $\mathbb C^n$ of dimension $e\le d$, the tuple $V^*XV$ lies in $\cD_A(e),$ where $V:M\to \mathbb C^n$ is the inclusion map. 
\end{lemma}

\begin{proof}
 To prove the non-trivial direction,  let a vector $v \in\mathbb C^d\otimes \mathbb C^n$ be given. Write  $v=\sum_{j=1}^d e_j\otimes v_j$, where $\{e_1,\dots,e_d\}$ is an orthonormal basis for $\mathbb C^d$. Let $M$ denote the span of $\{v_1,\dots,v_d\}$. Thus $M$ has dimension $e\le d$.  Let $V$ denote the inclusion of $M$ into $\mathbb C^n$.  Since $V^*XV\in \cD_A(e)$ by assumption, 
\[
  \langle L_A(X) v,v\rangle = \langle L_A(V^*XV)v,v\rangle \ge 0
\]
and the desired conclusion follows. 
\end{proof}

Before proceeding we address a technical point related to the Kronecker product that occurs in the following lemma. Note that 
for any $B_1, B_2 \in M_\ell(\C)$ and $Z \in M_{\nu}(\C)$ we have the identity
\beq
\label{eq:tensordirectsumidentity}
(B_1 \oplus B_2) \otimes Z=(B_1 \otimes Z) \oplus (B_2 \oplus Z).
\eeq
On the other hand, while $Z\otimes (B_1\oplus B_2)\neq
(Z\otimes B_1)\oplus (Z\otimes B_2)$, the fact that these two expressions are unitarily equivalent suffices for our arguments.
In fact, there is a permutation matrix, often called the canonical shuffle,
$\Pi_{\ell, \nu}\in M_{\nu\ell}(\C)$
 such that $B \otimes Z= \Pi_{\ell, \nu}^* (Z \otimes B) \Pi_{\ell, \nu}$
for any matrices $B \in M_\ell(\C)$ and $Z \in M_{\nu}(\C)$. 
We write $B\otimes Z\cs Z\otimes B$.

\begin{lemma}
\label{lem:lambdaball}
 {Suppose $\cD_A$ is matrix balanced, closed with respect to direct sums  and} $\Lambda=\La_F$ is a homogeneous linear pencil. 
\begin{enumerate}[label={\rm(\roman*)}]
 \item \label{it:0} If $\|\Lambda(X)\|>1$ for all $X\in M_d(\bbF)^g\setminus \cD_A(d)$, then, $\|\Lambda(Y)\|>1$ for each $1\le e\le d$ and 
    $Y\in M_e(\C)^g\setminus  \cD_A(e)$.
 \item \label{it:notinA} If $\|\Lambda(X)\|> 1$ for all $X\in M_d(\bbF)^g \setminus  \cD_A(d)$, then $\|\Lambda(X)\|> 1$ for all $X\notin \cD_A$.
 \item \label{it:lambdaball} If  $\|\Lambda(X)\|\le 1$ for all $X\in \cD_A$ and $\|\Lambda(X)\|=1$ for all $X\in \partial \cD_A(d)$, then $\cD_A=\cD_E$, where
$
 E = \begin{pmatrix} 0 & F \\ 0 & 0 \end{pmatrix}.
$
\end{enumerate}
\end{lemma}

\begin{proof}
 To prove item \ref{it:0}, suppose $1\le e\le d$ and $Y\in M_e(\C)^g\setminus  \cD_A(e)$. Thus $L_A(Y)\not\succeq 0$. Let $0$ denote the tuple of zeros in $M_{d-e}(\C)^g$ and let  $X=Y\oplus 0.$ Now $X\notin\cD_A(d)$ since $L_A(X) \cs L_A(Y)\oplus I\not\succeq 0.$  By hypothesis, $\|\Lambda(X)\|>1$. But $\Lambda(X)\cs \Lambda(Y)\oplus 0$. Hence, $\|\Lambda(Y)\|>1$.

By item \ref{it:0}, to prove item \ref{it:notinA} it may be assumed that $\|\Lambda(X)\|>1$ for all $1\le e\le d$ and $X\in M_e(\C)^g\setminus \cD_A(e)$.
 Let $n$ and $Y\in M_n(\bbF)^g \setminus \cD_A(n)$ be  given. By Lemma \ref{lem:incDA}, there is a subspace $M$ of
  dimension $e\le d$ such that, $X=V^*YV\notin \cD_A(e)$, where $V$ is  the inclusion of $M$ into $\bbF^d$.  Hence, by assumption,
 $\|\Lambda(X)\|>1$.  Hence there is a unit vector $v\in \mathbb C^N\otimes M\subset \mathbb C^N\otimes \C^n$, where $N$ is the size of the pencil $\Lambda$, such that
$\|\Lambda(X)v\|>1$.  Consequently,
\[
 1<\|\Lambda(X)v\| = \|(I\otimes V)^* \Lambda(Y)(I\otimes V)v\| \le \|\Lambda(Y)\|\, \|v\| =\|\Lambda(Y)\|.
\]

  To prove item \ref{it:lambdaball}, first note that the hypotheses 
immediately imply $\cD_A\subset \cD_E$. To prove the reverse inclusion, 
observe, if $X\notin \cD_A(d)$, then there is an $0<r< 1$ such that
$rX\in \partial \cD_A(d)$ (since $0$ is in $\cD_A(d)$ and $\cD_A(d)$
is convex) and hence $\|\Lambda(X)\|=\frac{1}{r}> 1$. Thus, if
$X\notin \cD_A(d)$, then $\|\Lambda(X)\|>1$.  It follows from item
\ref{it:notinA} that  $X\notin \cD_A$ implies
$X\notin \cD_E$. Hence $\cD_E\subset\cD_A$ and the proof is complete.
\end{proof}

\subsection{Free Circular Free Spectrahedra}
  The final part of Theorem \ref{theorem:circularmain}, stated in a somewhat different form below as Corollary \ref{cor:mainball}, is proved in this subsection.

\begin{theorem}[Theorem~\ref{theorem:circularmain}~\ref{it:freecirc}]
 \label{thm:EWball}
 If $\cD_A$ is a free circular spectrahedron, then there exists a homogeneous linear pencil $\Lambda$ such that $\|\Lambda(X)\|\le 1$ if and only if $X\in \cD_A$. Moreover, $\Lambda$ is the direct sum of at most $d^2$ homogeneous linear pencils of size (at most) $d$.
\end{theorem}

\begin{cor}
 \label{cor:mainball}
  Suppose $A\in M_d(\F)^g$.   If $L_A$ is a minimal defining pencil for $\cD_A$ and $\cD_A$ is free circular, then there exists positive integers $s,t$ such that $s+t=d$ and a $g$-tuple $F$ of $s\times t$  matrices with entries from $\F$ such that, 
\[
 A \simu \begin{pmatrix} 0 & F \\ 0 & 0 \end{pmatrix}.
\]
\end{cor}

\begin{proof}[Proof of Corollary~\ref{cor:mainball}]
 By Theorem \ref{thm:EWball}, there exist positive integers $m,n$ and a tuple $G$ of $m\times n$ matrices  such that $\cD_A=\cD_B,$ where
\[
 B =\begin{pmatrix}0& G\\ 0& 0\end{pmatrix}.
\]
 In particular, the size of $B$ is $(m+n)\times (m+n)$.  Next observe,  without loss of generality, it may be assumed that $\ker(G)=\{0\}=\ker(G^*)$. 

There is a reducing subspace $\mathcal E\subset \F^m\oplus \F^n$ such that, letting $E$ denote the restriction of $B$ to $\mathcal E$, the monic linear pencil $L_E$ is minimal defining for $\cD_A$ 
(cf.~Proposition \ref{prop:Gsatz}).
   Hence by loc.~cit.~$A$ and $E$ are unitarily equivalent.  Let $\mathcal G$ denote the projection of $\mathcal E$ onto the first coordinate and $\mathcal G_*$ denote the projection onto the second coordinate. Thus $\mathcal E\subset \mathcal G\oplus \mathcal G_*$. On the other hand, since $\mathcal E$ is reducing for $E$,
\[
 G_j^* G_j \mathcal G_* =  B_j^* B_j \mathcal E = \begin{pmatrix} 0 & 0 \\ 0 & G_j^* G_j \end{pmatrix} \mathcal E   \subset \mathcal E.
\]
 Hence each $G_j^*$ maps $\mathcal G_*$ into $\mathcal G_*$ and  $\sum_{j=1}^g G_j^* G_j \mathcal G_* \subset \mathcal G_*$. On the other hand, since $\sum G_j^* G_j$ does not have a kernel, it follows that the span of the subspaces $G_j^* \mathcal G_*$ is precisely $\mathcal G_*$.  Thus $\mathcal G_*\subset \mathcal E$. Likewise $\mathcal G\subset \mathcal E$. Hence $\mathcal E =\mathcal G\oplus \mathcal G_*$ and thus,
\[
  E = W^* B W = \begin{pmatrix} 0 & VGV_* \\ 0 & 0 \end{pmatrix},
\]
 where $W$ is the inclusion of $\mathcal E$ into $\F^{m+n}$ and $V$ and $V_*$ are the inclusions of $\mathcal G$ and $\mathcal G_*$ into $\F^m$ and $\F^n$ respectively.
 \end{proof}

 The proof of Theorem \ref{thm:EWball} rests on two preliminary lemmas.  Given a vector $v=\sum_{k=1}^d e_k\otimes v_k \in \bbF^d\otimes \bbF^n$ and matrix $\eta\in \Fdd$, let
\[
[ \eta,v] = \sum_{s=1}^d e_s\otimes \big(\sum_{k=1}^d \eta_{s,k}v_k\big) \in \bbF^d\otimes\bbF^n = \bbF^{nd}.
\]
 A pair $(X,v)\in M_n(\bbF)^g\times (\bbF^{nd}\setminus\{0\})$ is in the \df{detailed boundary} of $\cD_A(n)$ if $X\in \cD_A$ and  $L_A(X)v=0$.

\begin{lemma}
\label{lem:finite}
 Fix positive integers $n,N$ and suppose $(X^j,v^j)\in M_n(\bbF)^g \otimes \bbF^{nd}$ are in the detailed boundary  of $\cD_A(n)$ for $1\le j\le N$. Write, $v^j \in \bbF^{nd} = \bbF^d\otimes\bbF^n$ as
\beq\label{eq:vten}
 v^j = \sum_{k=1}^d  e_k\otimes v^j_k.
\eeq
Let $\mathcal P$ denote the subspace of $\Fdd$ consisting of those matrices $c$ such that $[ c, v^j] =0$ for all $1\le j\le N.$ (In this context, we identify $\Fdd$ with $\bbF^{d^2}$ or equivalently endow $\Fdd$ with the Hilbert-Schmidt norm.) There exists a homogeneous linear pencil $\Lambda$ of size $\size$, an $\ell$ and a nonzero matrix $\eta\in \mathcal P^\perp$ such that $\|\Lambda(Z)\|\le 1$ for all $Z\in \cD_A,$ and such that $[\eta,v^\ell]\ne 0$  and  if $[ \eta,v^j] \ne 0$, then $\|\Lambda(X^j)\|=1$. 
\end{lemma}

\begin{proof}
Let $\{\epsilon_j \}_{j=1}^N$ be the standard orthonormal basis for $\C^N$ and let $\mathcal{E}^j=\epsilon_j \epsilon_j^* \in \C^{N \times N.}$
Let  $Y=\sum_{j=1}^N X^j \otimes \cEE^j$. Since $Y$ is unitarily equivalent to $\oplus_{j=1}^N X^j,$ it follows that $Y\in \cD (nN)$.  
Let  $v=\sum_{j=1}^N v^j \otimes \epsilon^j \in \bbF^{ndN}$. Thus,   $v=\sum_{k=1}^d e_k\otimes v_k$, where, for $1\le k\le d$, 
\[
 v_k =  v^j_k \otimes \epsilon_j  \in \bbF^{nN}. 
\]
Let $M$ denote the span of $\{v_k: 1\le k \le d\}$ as a subspace of $\bbF^{nN}$ and let $m$ denote the dimension of $M.$ In particular, $m \leq d$. 
Let $V$ denote the inclusion  of $M$ into $\bbF^{nN}$ and  let $Z=V^*YV$.  Note that 
$Z\in \cD_A(m)$ since $\cD_A$ is matrix convex and $V$ is an isometry. Observe that 
\begin{equation}
\label{eq:calE}
\begin{split}
 \Lambda_A(Y) = & \sum_{k=1}^g A_k \otimes Y_k =\sum_{k=1}^g A_k\otimes \big (\sum_{j=1}^N (X^j_k  \otimes \mathcal E^j) \big ) \\
 = & \sum_{j=1}^N \big (\sum_{k=1}^g A_k\otimes X^j_k \big )\otimes \mathcal E^j = \sum_{j=1}^N \Lambda_A(X^j)\otimes \mathcal E^j.
\end{split}
\end{equation}
It follows from equation \eqref{eq:calE}, that 
\[
 \langle L_A(Z)v,v\rangle  = \langle L_A(Y)v,v\rangle = \sum_{j=1}^N \langle L_A(X^j)v^j,v^j\rangle  = 0.
\]
 Thus $Z$ boundary of $\cD_A(m)$.  By Proposition \ref{prop:sharp}, there is a homogeneous linear pencil $\Lambda$ of size $m \leq d$ (and without loss of generality we take $\Lambda$ of size $d$)  such that $\|\Lambda(X)\|\le 1$ for all $X\in \cD_A$ and $\|\Lambda(Z)\|=1$.  Thus, there is a unit vector $\gamma\in \bbF^d\otimes  M$ such that $\|\Lambda(Z)\gamma\|=1$.  It follows that $\gamma$ is in the span of $\{e_s\otimes v_k:1\le s, k\le d\}$; i.e, $\gamma\in \C^d\otimes M$.  In particular, there is a $\mu\in\Fdd$ such that $\gamma = \sum_{s=1}^d e_s\otimes (\sum_{k=1}^d \mu_{s,k} v_k)=[ \mu, v]$. Let $\gamma^j = [ \mu,v^j]$.  Thus 
$\gamma = \sum_{j=1}^N \gamma^j \otimes \epsilon^j$ and
 $\gamma^j\ne 0$ if and only if $[ \mu,v^j] \ne 0$.  Estimate, using equation \eqref{eq:calE},
\[
 \begin{split}
 1 & =   \|\Lambda(Z)\gamma\|^2 
   = \| \Lambda(V^*YV)\gamma\|^2 
   =   \| (I\otimes V^*) \Lambda(Y) \gamma\|^2 \\
   &\le \|\Lambda(Y)\gamma\|^2
     =  \sum_{j=1}^N \|\Lambda(X^j)\gamma^j \|^2 
   \le \sum_{j=1}^N \|\gamma^j\|^2 = 1.
\end{split}
\]
It follows that $\|\Lambda(X^j)\gamma^j\|=\|\gamma^j\|$ for all $1\le j\le N$.  Moreover, there exists an $\ell$ such that $\|\gamma^\ell\|\ne 0.$  Equivalently, $[ \mu,v^\ell] \ne 0$. Furthermore, $\|\Lambda(X^\ell)\|=1$  for each such $\ell$.  To complete the proof, let $\eta$ denote the projection of $\mu$ onto $\mathcal P^\perp$.  Since $[ \eta,v^j] =[ \mu, v^j] =\gamma^j$, it follows that $[ \eta,\gamma^j] \ne 0$ implies $\|\Lambda(X^j)\|=1$. Finally,  $[ \eta,v^\ell] \ne 0.$ 
\end{proof}

\begin{lemma}
\label{lem:countable}
 Fix a positive integer $n$ and suppose $(X^j,v^j)$ is a sequence from the detailed boundary of $\cD_A(n)$. Write, $v^j \in \bbF^d\otimes\bbF^n$ as in \eqref{eq:vten}.
Let $\mathcal P$ denote the subspace of $\Fdd$ consisting of those matrices $c$ such that $[ c,v^j]  =0$ for all $j$.

There exists a homogeneous linear pencil $\Lambda$ of size $\size$  and a nonzero matrix $\eta\in \mathcal P^\perp$ such that $\|\Lambda(Z)\|\le 1$ for all $Z\in \cD_A$  and such that if $[ \eta,v^j ] \ne 0 $, then $\|\Lambda(X^j)\|=1$. In particular, there is a $j$ such that $[ \eta,v^j] \ne 0$.
\end{lemma}

\begin{proof}
For positive integers $N$, 
let $\mathcal P_N$ denote the subspace of $\Fdd$ consisting of those matrices $c$ such that $[ c,v^j]  =0$ for $1\le j\le N$.  Hence, $\mathcal P_1 \supseteq \mathcal P_2 \supseteq \cdots$ and $\mathcal P=\cap_{N=1}^\infty \mathcal P_N$.

By Lemma \ref{lem:finite}, for each $N$ there exists a homogeneous
linear pencil $\Lambda^N$ of size $\size$ and a unit vector (matrix
of Hilbert-Schmidt norm one) $\eta^N \in \mathcal P^\perp$ such that
$\|\Lambda^N(X)\|\le 1$ for all $X\in \cD_A$ and, if $1\le j\le N$ and
$[ \eta^N, v^j ] \ne 0$, then $\|\Lambda^N(X^j)\|=1$. Write,
\[ 
 \Lambda^N(x)  = \sum_{j=1}^{g}\Lambda^N_j x_j.
\]
Since $\cD_A$ contains a free neighborhood of $0$, there is a uniform
bound on the norms of the matrices $\{\Lambda^N_j: j,N\}$. It follows
that there are subsequences $(\Lambda^{N_\ell})_\ell$ and
$(\eta^{N_\ell})_\ell$ converging to some $\Lambda$ and $\eta$
respectively. In particular, $\|\Lambda(X)\|\le 1$ for all
$X\in\cD_A$. Since $\eta^{N_\ell} \in \mathcal P_M$ for $N_\ell \ge M$
and since $\mathcal P_M$ is a (closed) subspace of $\Fdd$, it follows
that $\eta \in P_M$ and consequently $\eta \in \cP^\perp$ is a unit
vector. Hence there is a $j$ such that $[\eta,v^j]\ne 0$.  Thus $[
\eta^{N_\ell},v^j] \ne 0$ for large enough $\ell$. For such $\ell$ it
follows, from Lemma \ref{lem:finite}, that
$\|\Lambda^{N_\ell}(X^j)\|=1$ and hence $\|\Lambda(X^j)\|=1$.
\end{proof}

\begin{proof}[Proof of Theorem~\ref{thm:EWball}]
  Let $J_0$ denote a countable set and choose a dense subset $\{X^j:j\in
  J_0\}$ of the boundary of $\cD_A(d)$ indexed by $J_0$. For each
  $j\in J_0$ there is a unit vector $v^j$ such that $(X^j,v^j)$ is in
  the detailed boundary of $\cD_A(d)$. Write $v^j = \sum_{k=1}^g
  e_j\otimes v^j_k$.  Let $\mathcal P_0$ denote those vectors
  $c\in\Fdd$ such that $[ c,v^j] =0$ for all $j\in J_0$.  By Lemma
  \ref{lem:countable}, there exists a linear pencil $\Lambda^1$ of
  size $\size$ and a unit vector $\eta^1\in \mathcal P_0^\perp$ such
  that $\|\Lambda(Z)\|\le 1$ for all $m$ and $Z\in\cD_A(m)$ and
  $\|\Lambda(X^j)\|=1$ for each $j\in J_0$ such that $[ \eta^1,v^j]
  \ne 0$. Moreover, there is a $j_0\in J_0$ such that $[
  \eta^1,v^{j_0}] \ne 0$.  Let $J_1$ denote those indices $j\in J_0$
  such that $[ \eta^1,v^j] =0$. Thus, $\|\Lambda^1(X^j)\|=1$ for
  $j\notin J_1$ and $J_1$ is a proper subset of $J_0$ since $j_0\in
  J_0$, but $j_0\notin J_1$. If $J_1$ is empty, the proof is nearly
  complete. Otherwise, let $\mathcal P_1$ denote the subspace of
  vectors $c\in \Fdd$ such that $[ c,v^j] =0$ for all $j\in
  J_1$. Observe that $\eta^1 \in \cP_1$, but $\eta^1\notin \cP_0$
  since $[ \eta^1, v^{j_0}] \ne 0.$ Therefore $\cP_0$ is a proper
  subspace of $\cP_1$. For the collection $\{(X^j,v^j): j\in J_1\}$
  there exists a homogeneous linear pencil $\Lambda^2$ of size $\size$
  and unit vector $\eta^2 \in \mathcal P_1^\perp$ such that if $j\in
  J_2$ and $[ \eta_2,v^j] \ne 0$, then $\|\Lambda^2(X^j)\|=1$ and,
  letting $J_2$ denote those $j\in J_1$ such that $[ \eta^2,v^j] =0$,
  the subspace $\cP_2$ consisting of those $c\in \Fdd$ such that $[
  c,v^j] =0$ for all $j\in J_2$ properly contains $\cP_1$.
  Recursively define $\cP_N$ and observe $\Fdd \supseteq \cP_N$. Since
  $\Fdd$ is finite dimensional this process terminates after $\rho\le
  d^2$ steps and produces
\begin{enumerate}[label={\rm(\roman*)}]
 \item a chain of subspaces $\cP_0\subsetneq \cP_1 \subsetneq \cdots \subsetneq \cP_\rho =\Fdd$ of $\Fdd$;
 \item a chain of subsets $J_0\supsetneq J_1\supsetneq \cdots \supsetneq J_\rho =\emptyset$; 
 \item  homogeneous linear pencils $\Lambda^r$ for $1\le r\le \rho$ of size $\size$ such that 
      $\|\Lambda^r(X)\|\le 1$ for $X\in\cD_A$ and $\|\Lambda(X^j)\|=1$ for each $j\in J_{r-1}\setminus J_r.$
\end{enumerate}
Let $\Lambda =\oplus_{r=1}^\rho \Lambda^r$.  Thus, by construction,
$\|\Lambda(X)\|\le 1$ for all $X\in \cD$ and $\|\Lambda(X^j)\|=1$ for
all $j\in J_0$. By continuity $\|\Lambda(Y)\|=1$ for all $Y$ in the
boundary of $\cD_A(d)$.  An application of Lemma \ref{lem:lambdaball}~\ref{it:lambdaball} completes the proof of the existence of
$\Lambda$. The bound $d^3$ follows since $\Lambda$ is the direct sum
of at most $d^2$ pencils each of size at most $\size$.
\end{proof}

\section{Free Polynomials  Invariant under Coordinate Unitary Conjugation}\label{sec:UnitaryConjPoly}
The main result of this section is 
Theorem \ref{theorem:invarpoly} characterizing monic free matrix polynomials
that are invariant under coordinate unitary conjugation.  The needed background on free polynomials and their evaluations are collected in the next subsection. Experts can skip straight to Subsection \ref{sec:IPs}.

\subsection{Words,  Free Polynomials and Evaluations}
We write $\langle x,x^* \rangle$ for the monoid freely generated by $x=(x_1, \dots x_g)$
 and $x^*=(x_1^*, \dots, x_g^*)$, i.e., $\langle x,x^* \rangle$ consists of {\bf words} in 
the $2g$ noncommuting letters $x_1, \dots , x_g,x_1^*, \dots x_g^*$ (including the
 {\bf empty word} $\emptyset$ which plays the role of the identity). Let $\C \langle x,x^* \rangle$ 
denote the associative $\C$-algebra freely generated by $x$ and $x^*$, i.e., the elements
 of $\C \langle x,x^* \rangle$ are polynomials in the freely noncommuting variables $x$ and
 $x^*$ with coefficients in $\C$. Its elements are called {\bf free polynomials}. The
 {\bf involution} $^*$ on $\C \langle x,x^* \rangle$ extends the complex conjugation
 on $\C$, satisfies $(x_i^*)^*=x_i$, reverses the order of words, and acts $\R$-linearly
 on polynomials. Polynomials fixed under this involution are {\bf symmetric}.
The length of the longest word in a free polynomial $f \in \C \langle x,x^* \rangle$ is the 
{\bf degree} of $f$ and is denoted by $\text{deg}(f)$ or $|f|$ if $f \in \langle x,x^* \rangle$.
 The set of all words of degree at most $k$ is $\langle x,x^* \rangle_k$, and $\C \langle x,x^* \rangle_k$ is 
the vector space of all free polynomials of degree at most $k$. 

Fix positive integers $v$ and $\ell$. {\bf Free matrix polynomials} - elements of 
$\C^{\ell \times v} \langle x,x^* \rangle= \C^{\ell \times v} \otimes \C \langle x,x^* \rangle$; 
i.e., $\ell \times v$ matrices with entries from $\C \langle x \rangle$ - will play a role in what
 follows. Elements of $\C^{\ell  \times v} \langle x \rangle$ are represented as 
\beq
\label{eq:freepolydef}
p(x)=\sum_{w \in \langle x,x^* \rangle} B_w w(x) \in \C^{\ell \times v} \langle x,x^* \rangle
\eeq
where the sum is finite, $B_w \in \C^{\ell \times v}$,
and $w(x)$ runs over words in $x$ and $x^*$. The involution $^*$ extends to matrix polynomials by
\[
p(x)^*=\sum_{w \in \langle x, x^* \rangle} B_{w}^* w(x)^* \in \C^{v \times \ell} \langle x,x^* \rangle.
\]
If $v= \ell$ and $p(x)^*=p(x)$, we say $p$ is symmetric. Additionally if $p(0)=I$, we say $p$ is \df{monic}.\looseness=-1

If $p \in \C \langle x,x^* \rangle$ is a free 
polynomial and $X \in M_n(\C)^g$, then the evaluation $p(X) \in M_{n}(\C)$ 
is defined in the natural way by replacing $x_i$ by $X_i$,
$x_i^*$ by $X_i^*$ and sending the empty word to
 the appropriately sized identity matrix. Such evaluations produce (all) finite dimensional
 $*$-representations of the algebra of free polynomials. 
Polynomial evaluations extend to matrix polynomials by evaluating entrywise. 
That is, if $p$ is as in \eqref{eq:freepolydef}, then
\[
p(X)=\sum_{w \in \langle x,x^* \rangle} B_w \otimes w(X) \in C^{\ell \times v}\otimes M_n(\C). 
\]
Note that if $p \in M_d(\C) \langle x, x^* \rangle$ is symmetric and
 $X \in M_n(\C)^g$, then $p(X) \in  M_d(\C)\otimes M_n(\C)=M_{dn}(\C)$ is a self-adjoint 
 matrix. 

\subsection{Invariant Polynomials}
\label{sec:IPs}
In this subsection we prove Theorem \ref{theorem:invarpolyintro} stated below in a self contained fashion for the reader's convenience. Write $A\simu B$ to indicate the matrices $A$ and $B$ are unitarily equivalent.

 \begin{theorem} 
\label{theorem:invarpoly}
Suppose $p$ is a monic free $d \times d$ matrix polynomial. For each $n$ and for each $g$-tuple of unitaries $U=(U_1, \dots, U_g) \in M_n(\C)^g$ there exists a unitary $W$ such that for all 
$X\in M_n(\C)^g$, 
\[
 p(U_1^*X_1U_1,\dots,U_g^* X_g U_g) =W^* p(X_1,\dots,X_g) W 
\]
 if and only if 
\beq
 \label{eq:invarpoly}
p(x)\simu 
p_1(x_1) \oplus \cdots \oplus p_g(x_g); 
\eeq
i.e.,  $p$ is (up to unitary equivalence) a direct sum of univariate matrix polynomials. 
\end{theorem}

The following lemma  is needed in the proof of Theorem \ref{theorem:invarpoly}.

\begin{lemma}
\label{lemma:polydirectsumnonsym}
Suppose $p(x)=\sum_{i=1}^g p_i (x_i)$ is a free 
$d\times d$
matrix polynomial, where 
\beq\label{eq:62}
p_i(x_i) p_j (x_j)= p_i(x_i)^* p_j (x_j) =p_i(x_i)p_j (x_j)^* =0
\eeq
 whenever $i \neq j$. Then there exists a unitary $U$ such that 
\beq\label{eq:Ublock}
U^*p(x)U= 
\hat{p}_1 (x_1) \oplus \cdots \oplus \hat{p}_g(x_g). 
\eeq
for some free matrix polynomials $\hat{p}_j$ each in the variables $x_j,x_j^*$ alone.
\end{lemma}

\begin{proof}

Suppose \eqref{eq:62} holds
 whenever $i \neq j$. Using the notation $w_i (x_i)$ to denote words in $x_i$ and $x_i^*$, write 
 \[p_i (x_i)=\sum_{w_i} A_{w_i} w_i (x_i).\]
  Then
\beq\label{eq:long}
p_i (x_i) p_j(x_j)= 
\sum_{w_i,w_j} A_{w_i} A_{w_j} w_i (x_i) w_j (x_j)=0.
\eeq
Note that, if $w_i(x_i),v_i (x_i), w_j(x_j),v_j(x_j)$ are words in $x_i$ and $x_j$, respectively, then we have $w_i (x_i) w_j(x_j)=v_i(x_i) v_j(x_j)$ if and only if $w_i(x_i)=v_i(x_i)$ and $w_j (x_j)= v_j(x_j)$. This implies that each monomial appears on the right hand side of \eqref{eq:long} exactly once. 
It follows that $A_{w_i} A_{w_j} =0$ for all $w_i, w_j$ whenever $i \neq j$. 
Similarly,
\beq\label{eq:longadjoint}
p_i (x_i)^* p_j(x_j)= 
\sum_{w_i,w_j} A_{w_i}^* A_{w_j} w_i (x_i)^* w_j (x_j)=0.
\eeq
Since each monomial appears on the right hand side of \eqref{eq:longadjoint} exactly once,
it follows that $A_{w_i}^* A_{w_j} =0$ for all $w_i, w_j$ whenever $i \neq j$.
Furthermore, 
\beq\label{eq:longgrammian}
p_i (x_i) p_j(x_j)^*= 
\sum_{w_i,w_j} A_{w_i} A_{w_j}^* w_i (x_i) w_j (x_j)^*=0.
\eeq
It follows that $A_{w_i} A_{w_j}^* =0$ for all $w_i, w_j$ whenever $i \neq j$. 

Let $\cA_j$ denote the finite dimensional (non-unital) $C^*$-algebra generated by  
\[\{A_{w_j}: w_j \mathrm{\  is \ a \ word \ in \ } x_j,x_j^*\}.\]
 Then 
\beq\label{eq:kill}
\cA_j \cA_\ell = \{0\}\quad\text{ for $j\neq \ell$}.
\eeq 

Decompose $\C^d$ as a direct sum of invariant (hence reducing) subspaces  for $\cA_1$,
say
\[
\C^d= \cS_1\oplus \cdots\oplus \cS_m \oplus \cS_{m+1},
\]
where $\cA_1$ acts irreducibly on $\cS_j$ for $j\leq m$ and $\cA_1(\cS_{m+1})=0$.
From \eqref{eq:kill} it follows that $\cA_k$ for $k\geq2$ vanishes on $\cS_1,\ldots,\cS_m$.
In particular, $\cS_{m+1}= (\cS_1\oplus \cdots\oplus \cS_m )^\perp$ is invariant under $\cA_k$ for $k \geq 2$. Thus $p(x)= \hat{p}_1 (x_1) \oplus q(\hat{x})$ where $q$ is a free matrix polynomial depending only on $\hat{x}=(x_2, \dots x_g)$ and $\hat{x}^*=(x_2^*, \dots, x_g^*)$.
We can repeat the above consideration -- decomposing $\cS_{m+1}$ into a direct sum of reducing subspaces  for $\cA_2$, etc.
Tracking down all these decompositions yields the desired block form \eqref{eq:Ublock}.
\end{proof}

\begin{proof}[Proof of Theorem \ref{theorem:invarpoly}]
Let $x=(x_1, \dots, x_g)$ be a $g$-tuple of noncommuting letters and suppose $p$ is a monic  free $d \times d$ matrix polynomial that is invariant under coordinate unitary conjugation. Here $p$ is given by $p(x)=\sum_w B_w w(x)$. Call a monomial a (noncommutative) cross term if it contains a product of the form $x_i x_j$ or $x_ix_j^*$ or $x_i^*x_j$ or $x_i^*x_j^*$ where $i \neq j$.
Our immediate goal is to show that $p(x)$ does not have any cross terms.

To this end, let $\mathcal{C}_x$ be the
 set of all cross term monomials
and define the free matrix polynomials $p^{\mathrm{ncr}}$ and $p^{\mathrm{cr}}$ by 
\beq
\label{eq:pCRpNCRdefinitions}
p^{\mathrm{ncr}} (x)=\sum_{w(x) \notin \cC_x} B_w w(x), \quad \quad \quad \quad
p^{\mathrm{cr}} (x)=\sum_{w(x) \in \cC_x} B_w  w(x).
\eeq
Here $p^{\mathrm{ncr}} (0)=I_d$ and $p^{\mathrm{cr}} (0)=0_d$. With this notation,
\beq
\label{eq:pCRpNCRform}
p(x)=p^{\mathrm{ncr}} (x)+p^{\mathrm{cr}} (x).
\eeq
To show $p$ has no cross terms we will show $p(x)=p^{\rm ncr} (x)$. 

Define $\tilde{x}_1, \dots, \tilde{x}_g$ by
\[
\tilde{x}_1=
x_1 \oplus 0 \oplus \cdots \oplus 0, \quad
\tilde{x}_2=0 \oplus x_2 \oplus \cdots\oplus 0, \ldots,\quad
   \tilde{x}_g=0\oplus \cdots\oplus 0\oplus x_g, \]
Choose permutation matrices $U_i$ so that
$
U_i^* \tilde{x}_i U_i=
x_i\oplus 0 \oplus \cdots \oplus 0
$
for all $i$. 

Recall the canonical shuffle discussed before Lemma \ref{lem:lambdaball}. We use it again here dealing with polynomials. Namely,
if $f$ is a $d \times d$ free matrix polynomial then the notation $f(x) \cs h(x)$ means that for all $n$ and for all $X \in M_n(\C)^g$ there exists a matrix $\hat{\Pi}_n$ that is a product of direct sums of canonical shuffles such that $\hat{\Pi}_n^* f(X) \hat{\Pi}_n=h(X)$.

Consider $p(\tilde{x})=p(\tilde{x}_1, \dots, \tilde{x}_g)$. Since $\tilde{x}_i \tilde{x}_j=0=\tilde{x}_i \tilde{x}_j^*=\tilde{x}_i^* \tilde{x}_j$ whenever $i \neq j$ we  see that
\[
q(x)=p(\tilde{x}) \cs q_1 (x_1) \oplus \dots \oplus q_g (x_g)
\]
where the $q_i$ are monic matrix polynomials each depending only on $x_i$ and $x_i^*$. Furthermore,
\[
p(U_1^* \tilde{x}_1 U_1, \dots, U_g^* \tilde{x}_g U_g) \simu  p(x) \oplus p(0) \oplus \dots \oplus p(0) =p(x) \oplus I_{d(g-1)}=\sum_w (B_w \oplus 0_{d(g-1)}) w(x).
\]

Fix $n$ and consider the evaluations $p(X) \oplus I_{nd(g-1)}$ and $q(X)$ on $g$-tuples of matrices $X \in M_n(\C)^g$. By assumption there exists a unitary $V_n$ depending only on our permutation matrices $U_i$ and on $n$ such that 
\beq
\label{eq:qpequality}
V_n^* q(X) V_n=p((U_1 \otimes I_n)^* \tilde{X}_1 (U_1 \otimes I_n), \dots, (U_g \otimes I_n)^* \tilde{X}_g (U_g \otimes I_n))
=p(X)\oplus I_{nd(g-1)}
\eeq
for all $X \in M_n(\C)^g$.

Define the $n \times n$ matrix $\cX^n_k$ by $\cX^n_k=(\cX^n_{k,ij})_{ij}$ for $1\leq k\leq g$. Here the $\cX^n_{k,ij}$ are commuting variables and $\cX^n_k$ is called a generic matrix. Define the $g$-tuple of $n \times n$ matrices $\cX^n$ by $\cX^n=(\cX_1^n, \dots, \cX_g^n)$. We say a word in the commuting letters $\{\cX^n_{k,ij}\}_{i,j,k}$ and $\{(\cX^n_{k,ij})^*\}_{i,j,k}$ is a commutative cross term if it contains a product of the form  $\cX^n_{k,ij} \cX^n_{\ell,rs}$ or $\cX_{k,ij}^n (\cX_{\ell,rs}^n)^*$ with $k \neq \ell$. Then \eqref{eq:qpequality} is equivalent to
\beq
 \label{eq:qpGenericEquality}
V_n^* q(\cX^n) V_n=p(\cX^n)\oplus I_{nd(g-1)}.
\eeq
We next show that the entries of $p(\cX^n)$ have no commutative cross terms.

Since $q(x)$ contains no cross terms it follows that for all $n$ the entries of $q(\cX^n)$ contain no commutative cross terms. Furthermore, the entries of $V_n^* q(\cX^n) V_n$ are linear combinations of the entries of $q (\cX^n)$ so it follows that for all $n$ the entries of $V_n^* q(\cX^n) V_n$ contain no commutative cross terms. Using  \eqref{eq:qpGenericEquality} we conclude that for all $n$ the entries of $p(\cX^n)\oplus I_{nd(g-1)}$, and hence the entries of of $p(\cX^n)$, contain no commutative cross terms.

Since the entries of $p(\cX^n)$ have no commutative cross terms we know from equation \eqref{eq:pCRpNCRform} that the entries of
$
p^{\mathrm{ncr}} (\cX^n)+p^{\mathrm{cr}} (\cX^n)
$
have no commutative cross terms. If a monomial $w(x)$ is not a cross term, then none of the entries of $w(\cX^n)$ are commutative cross terms. Therefore none of the entries of $p^{\mathrm{ncr}} (\cX^n)$ are commutative cross terms. Since $p(\cX^n)$ has no commutative cross terms this implies that none the entries of $p^{\mathrm{cr}}$ cannot be commutative cross terms. We conclude
$p^{\mathrm{cr}} (\cX^n)=0_{nd \times nd}$
and therefore
\beq
\label{eq:pNCRisp}
p^{\rm ncr} (\cX^n)=p(\cX^n).
\eeq

Equation \eqref{eq:pNCRisp} holds for all $n$, so we obtain that for all $n$ and for all $g$-tuples of $n \times n$ matrices $X$ we have the equality
\beq
\label{eq:pNCRisspEval}
p^\mathrm{ncr} (X)=p(X).
\eeq
Since equation \eqref{eq:pNCRisspEval} holds for all $n$, we conclude 
\beq
\label{eq:rhoEqualsp}
p^\mathrm{ncr} (x)=p(x).
\eeq
Therefore, $p$ has no cross terms, as claimed.

 Now $p$ can be written $p(x)=I+\sum_{i=1}^g p_i (x_i)$ where $p_i(0)=0$. Additionally, since $p$ is invariant under coordinate unitary conjugation
 it follows that $p^2$ defined by 
\beq
\label{eq:coordinvarsquared}
p^2 (x)=I+2\sum_i p_i(x)+\sum_{i,j} p_i (x_i) p_j (x_j)
\eeq
 is also invariant under coordinate unitary conjugation and hence $p^2$
 cannot have any cross terms. Thus equation \eqref{eq:coordinvarsquared}  implies that $p_i(x_i) p_j(x_j)=0$ whenever $i \neq j$. 

Additionally, since $p$ is invariant under coordinate unitary conjugation, given any unitaries $U_i \in M_{n}(\C)$ there exists some unitary $U \in M_{nd}(\C)$ depending only on the $U_i$ such that for any $X \in M_n(\C)^g$ we have
\beq\label{eq:4s1}
p(X_1, \dots X_g)=U^*p(U^*_1 X_1 U_1, \dots , U^*_g X_g U_g) U.
\eeq
It immediately follows that
\beq\label{eq:4s2}
p(X_1, \dots X_g)^*=U^* p(U_1^* X_1 U_1, \dots, U_g^* X_g U_g)^* U.
\eeq
These two equations imply
\beq
\label{eq:coordinvarprods}
\begin{array}{ccc}
p(x)^*p(x)=I+\sum_i p_i(x)+\sum_i p_i(x)^*+\sum_{i,j} p_i (x_i)^* p_j (x_j), \\
p(x)p(x)^*=I+\sum_i p_i(x)+\sum_i p_i(x)^*+\sum_{i,j} p_i (x_i) p_j (x_j)^* \\
\end{array}
\eeq
are also invariant under coordinate unitary conjugation and therefore have no cross terms. Therefore $p_i (x_i) p_j (x_j)^*=p_i (x_i)^* p_j (x_j) =0$ whenever $i \neq j$. 

It follows from Lemma $\ref{lemma:polydirectsumnonsym}$  that there exists a unitary $V \in M_d(\C)$ such that
\[
V^* p(x) V=
 \hat{p}_1 (x_1) \oplus \cdots \oplus
 \hat{p}_g(x_g),
\]
where the $\hat{p}_i$ are monic free matrix polynomials in the variable $x_i$. Thus,  if $p$ 
is invariant under coordinate unitary conjugation, then equation \eqref{eq:invarpoly} holds.

 The converse is straightforward. If \eqref{eq:invarpoly} holds, then evidently $p$ is invariant under coordinate unitary conjugation.
 \end{proof}

\begin{remark}\rm
We say a free spectrahedron $\cD$ is invariant under coordinate unitary conjugation if
 $X\in \cD$ implies $(U_1^* X_1 U_1, \dots , U_g^* X_g U_g) \in \cD$ for all $X \in M_n(\C)^g$ 
and all unitaries $U_1, \dots , U_g \in M_{n}(\C)$.
  Suppose the symmetric monic linear pencil $L_A$ is minimal in defining a free spectrahedron $\cD_{A}$.
 It follows from Theorem \ref{theorem:invarpoly} and  \cite[Theorem 1.2]{HKM13} that $\cD_{A}$ 
is invariant under coordinate unitary conjugation if and only if there is a unitary $U$ so that
\[
U^* L_A(x) U= 
\bigoplus_{j=1}^g \big(I-A_j x_j-A_jx_j^*\big).
\]
\end{remark}

\appendix

\section{Free Circular Matrix Convex Sets are Operator Pencil Balls}\label{sec:FreeCircDomain}
In this section we characterize free circular subsets of $M(\bbF)^g.$  A subset $D \subseteq M(\bbF)^g$ is \df{free circular} if $UX \in \cD$ for each $n$, each $X \in \cD(n)$ and each $n \times n$ unitary matrix $U$.

\subsection{Properties of Free Circular Sets}
\label{sec:FreeCircProperties}
 A free set $\cD  \subseteq M(\bbF)^g$ is an \df{operator pencil ball} if there exists a Hilbert space $\mathcal{H}$ over $\bbF$ and a $g$-tuple $A \in \mathcal{B(H)}^g$ such that $X \in \cD$ if and only if
\[
 \|\Lambda_A(X)\|\le 1.
\]
(Observe that the formulas \eqref{eq:homPencil} -- \eqref{eq:LMIdomain} naturally extend to tuples of operators $A$.)
In particular, an operator pencil ball can be described as the positivity set of the symmetric operator pencil
\[
 \begin{pmatrix} I & \Lambda_A(x) \\ \Lambda_A(x)^* & I \end{pmatrix}.
\]
If  $\mathcal{H}$ is finite dimensional (so $\mathcal{B(H)}^g \cong M_d(\bbF)^g),$ the set  $\cD$ is a \df{matrix pencil ball}.

The main result of this section is 
Theorem \ref{thm:OperatorBall}. It shows that a free circular matrix convex free set containing a neighborhood of $0$   is an operator pencil ball, and is thus
 a free circular analog of the Effros-Winkler matricial 
Hahn-Banach separation theorem \cite{EW,HM12}.\looseness=-1

\begin{lem}\rm
\label{lem:pencilQ}
Suppose $\cC$ is matrix balanced, closed with respect to direct sums and contains $0$ in its interior and  $Q\in M_d(\C)^g.$  If $\|\Lambda_Q(X)\|\le 1$ for $X\in \cC$, then $\|\Lambda_Q(X)\|<1$ for $X$ in the interior of $\cC$. Conversely, if $\|\Lambda_Q(X)\|<1$ for $X$ in the interior of $\cC$, then $\|\Lambda_Q(X)\|\le 1$ for $X\in \cC$. 
\end{lem}

\subsection{States and Representations of Separating Linear Functionals}
Let $M_{\ell}(\bbF)_{\rm sa}$  denote self-adjoint elements of $M_{\ell}(\bbF)$ and suppose $\cS$ is a subspace of $M_\ell(\bbF)_{\rm sa}$. 
An \df{affine linear mapping} $f:\cS \to \mathbb R$ is a
 function of the form $f(x)=a_f +\lambda_f(x)$,
 where $\lambda_f:\cS\to \mathbb R$ is linear over $\R$ and $a_f\in\mathbb \R$.   The following
 lemma is a version of   \cite[Lemma 5.2]{EW}. 

\begin{lemma}
 \label{lem:cone}
  Suppose $\mathcal F$ is a convex set of affine linear
  mappings $f:\cS \to \mathbb \R$ and $\cT\subset \cS$ is  compact and  convex.
  If for each $f\in \mathcal F$
  there is a $\hh \in\cT$ such that $f(\hh)\ge 0$,
  then there is a $\hhs\in \cT$ such that
  $f(\hhs)\ge 0$ for every $f\in\mathcal F$.
\end{lemma}

\begin{proof}
 Each $f\in \mathcal F$ is continuous, a fact we will use freely.  For $f\in\mathcal F$, let
 \[
   B_f =\{\hh\in \cT: f(\hh)\ge 0\}\subset \cT.
 \]
  By hypothesis each $B_f$ is non-empty and
  it suffices to prove that
 $$
   \bigcap_{f\in\mathcal F} B_f \neq \emptyset.
 $$
  Since each $B_f$ is compact, it suffices to
  prove that the collection $\{B_f: f\in\mathcal F\}$
  has the finite intersection property.  Accordingly,
  let $f_1,\dots,f_m\in\mathcal F$ be given. Arguing
  by contradiction, suppose
$
  \bigcap_{j=1}^m B_{f_j} =\emptyset.
$
Define $F:\cS \to \mathbb R^m$  by
\[
  F(\hh)=(f_1(\hh),\dots,f_m(\hh)).
\]
 Then $F(\cT)$ 
 is both convex and compact because $\cT$
 is both convex and compact since $F$ is continuous.
 Moreover, 
  $F(\cT)$  does not intersect 
 \[
   \mathbb R^m_{\geq0}=\{x=(x_1,\dots,x_m): x_j\ge 0 \mbox{ for each } j\}.
 \]
 Hence there is a linear functional $\lambda:\mathbb R^m \to \mathbb R$
 such that $\lambda(F(\cT))<0$ and $\lambda(\mathbb R_{\geq0}^m) \ge 0$.
 There exists $\lambda_j$ such that
$
 \lambda(x) = \sum \lambda_j x_j.
$
  Since $\lambda(\mathbb R^m_{\geq0}) \ge 0$ it follows that each
  $\lambda_j\ge 0$ and, since $\lambda\ne 0$, there is a $k$ such that $\lambda_k>0$.  Without loss of generality, 
  it may be assumed that $\sum \lambda_j=1$. 
  Let
\[
  f=\sum \lambda_j f_j.
\]
 Since $\mathcal F$ is convex, 
 it follows that $f\in\mathcal F$. On the other hand,
 $f(T)=\lambda(F(T)).$  Hence 
 if $T\in \cT,$ then  $f(T)<0$.
 Thus, for this $f$ there does not exist
 a $T\in \cT$ such that $f(T)\ge 0$,
 a contradiction which completes the proof.
\end{proof}

\begin{lemma}
 \label{lem:cT}
  Let $\cC=(\cC(n))$ denote a matrix balanced subset of
  the graded set $M(\bbF)^g$ that is closed with respect to direct sums.  Let $n$ and an $\bbF$-linear functional 
  $\mathcal L :M_n(\bbF)^g \to \bbF$ be given. 
  If 
   $
   \operatorname{Re}( \mathcal L (X)) \le 1
  $
   for each $X\in\cC(n)$,
  then there exits positive semidefinite  $n\times n$ matrices  $\mfT_1$ and $\mfT_2$ each of trace norm one
  such that   for each $m$, each $Y\in \cC(m),$ and each pair $C=(\Cr,\Cl)$ of 
  $m\times n$ matrices 
 \[
   2 \operatorname{Re}(\mathcal L (\Cl^{\ast} Y \Cr)) \le \tr(\Cr \mfT_1 \Cr^{\ast}) + \tr(\Cl\mfT_2 \Cl^\ast).
 \]
\end{lemma}

\begin{proof}
   Let $\ell=2n$ and let 
\[
  \cT =\{T= T_1\oplus T_2 : T_j \in M_n(\bbF)_{\rm sa}, \, T_j\succeq 0, \mbox{ and } \tr(T_j)= 1\}.
\]
 In particular $\cT$ is a compact convex subset of the $\ell\times\ell$ matrices. 
 
 Given a positive integer $m,$ a  tuple $Y$ in $\cC(m)$
 and  $m\times n$ contraction 
matrices $\Cr,\Cl$, define
 $f_{Y,C}:M_n(\bbF)_{\rm sa} \oplus \mathbb M_n(\bbF)_{\rm sa} \to \mathbb R$ by 
\[
  f_{Y,C}(T_1\oplus T_2)=\sum_{j=1}^2 \tr(C_j T_j C_j^{\ast})  - 2 \operatorname{Re}(\mathcal L (\Cl^{\ast}Y\Cr)).
\]

  Now we show that the collection
  \[
  \mathcal F=\{f_{Y,C}:Y \in \cC(m),  \ C=(C_1,C_2) \ \mathrm{where} \ C_1,C_2 \in \bbF^{m \times n} \ \mathrm{are \ contractions\ and \ }  m,n \in \mathbb{N} \}
  \]
 is
  a convex set. Start with a positive integer $s,$ 
  nonnegative numbers $\lambda_1,\dots,\lambda_s$ with
  $\sum \lambda_j=1,$ and with $(Y_j,C_{j,1},C_{j,2})$ 
  for $j=1,\dots,s$ where $Y_j\in\cC(m_j)$ and $C_{j,p}$ are $m_j\times n$
  contraction matrices. Let $Z=\oplus Y_j$
  and let $F_p$ denote the (block) column matrix with entries
  $\sqrt{\lambda_j} C_{j,p}$. Then $Z\in \cC(m)$ where $m=\sum m_j$ and
\[
  F_p^*F_p =\sum \lambda_j C_{j,p}^*C_{j,p} \preceq \sum \lambda_j I =I.
\]
 Hence each $F_p$ is a contraction.
 By definition,
\[
\begin{split}
 \sum \lambda_j C_{j,2}^* Y_j C_1 & = F_2^* ZF_1,\qquad
 \sum \lambda_j \tr(C_{j,p} T_p C_{j,p}^*) = \tr(F_p T_pF_p^*).
 \end{split}
\]
 Therefore
\[
  \sum \lambda_j f_{Y_j,C_j}(\hh) = f_{Z,F}(\hh)
\]
so $\mathcal{F}$ is convex.

 Observe, for any $X\in \cC$ and  pair of matrices $\Cl$ and $\Cr$  (of the appropriate sizes) $\operatorname{Re} \cL(\Cl^* X \Cr) \le \|\Cl\|\,\|\Cr\|$. Now  let $f_{Y,C} \in\mathcal F$ be given.  Choose unit vectors $\gamma_j$ such that
\[
 \|C_j \gamma_j \| =\|C_j\|,
\]
  let $T_j=\gamma_j^* \gamma_j$ and finally $T=T_1\oplus T_2$. With these  notations,
\[
2 \operatorname{Re}(\cL( \Cl^{\ast} Y \Cr) \le 2\|C_1\|\, \|C_2\| \le \|C_1\|^2 +\|C_2^2\| =\tr(C_1\gamma_1\gamma_1^* C_1) + \tr(C_2\gamma_2 \gamma_2^* C_2^*)
\]
and thus,
\[
 f_{Y,C}(T) = \sum_{j=1}^2\tr(C_j \gamma_j\gamma_j^* C_j^*) - 2 \operatorname{Re}(\mathcal L (\Cl^{\ast} Y \Cr)) \ge 0.
\]

 Consequently, for each $f_{Y,C}$ there is a $T\in\cT$ such
 that $f_{Y,C}(T)\ge 0$. From 
 Lemma \ref{lem:cone}, there is a $\hhs \in \cT$ such that
 $f_{Y,C}(\hhs)\ge 0$ for every $Y$ and $C$. 
\end{proof}

\subsection{An Effros-Winkler Theorem for Free Circular Matrix Convex Sets}
 
 In this section we present the effective version of \cite[Proposition 3.5]{BVM}, i.e.,
Proposition \ref{prop:sharp}, restated here for the convenience of the reader as
  Proposition \ref{prop:2sharp}.

\begin{prop}
 \label{prop:2sharp}
 Let $\cC=(\cC(n))$ denote a  matrix balanced
 subset of the graded set  $M(\bbF)^g$ that 
 contains a free $\epsilon$-neighborhood of $0$ and is closed with respect to direct sums.
 If $\oX\in M_n(\bbF)^g$ is in the boundary of $\cC(n)$,
 then there is a tuple $Q\in M_n(\bbF)^g$ such that 
 $\|\Lambda_Q(Y)\|\le 1$ for all $m$ and
 $Y\in\cC(m)$  and such that  $\|\Lambda_Q(\oX)\|=1$. 
 Furthermore,
 if $Y$ is in the interior of $\cC$, then $\|\Lambda_Q(Y)\|<1$.
\end{prop}

\def\beq{\begin{equation}
}
\def\eeq{\end{equation}}

\begin{proof}
 By the usual Hahn-Banach separation theorem
 and the assumption that $\cC(n)$ contains
 an $\epsilon$-neighborhood of $0$,
 there is a linear functional $\cL: M_n(\bbF)^g \to \bbF$
 such that $\operatorname{Re}(\cL(\oX)) =  1 \ge  \operatorname{Re}(\cL(\cC(n)))$.
 
 From Lemma \ref{lem:cT} there exists positive semidefinite $n\times n$ matrices $T_1$ and $T_2$ of  trace norm one such that
$\sum_{p=1}^2 \tr(C_p T_p C_p^{\ast})- 2 \operatorname{Re}(\cL(\Cl^{\ast}Y\Cr)) \ge 0$  for each $m$, each pair of $m\times n$ contractions $\Cr,\Cl$, and each $Y\in\cC(m)$. Hence, by homogeneity,   for each $m$, each pair of $m\times n$ matrices $\Cl,\Cr$, and each $Y\in\cC(m)$,
\begin{equation}
\label{eq:absLMI}
 \sum_{p=1}^2 \tr(C_p T_p C_p^{\ast})- 2\operatorname{Re}(\cL(\Cl^{\ast}Y\Cr)) \ge 0
\end{equation}
 Note this inequality is sharp in the sense,
\begin{equation}
\label{eq:absLMI0}
\sum_{p=1}^2 \tr(T_p)- 2\operatorname{Re}(\cL(\oX)) = 0.
\end{equation}

 Let $\{\be_1,\dots,\be_g\}$ 
  denote the
 standard orthonormal basis for $\bbF^g.$ Thus, if $M$ is an $n\times n$ matrix, then $M\otimes \be_\ell=(M_1,\dots,M_g)\in M_n(\bbF)^g$ is the $g$-tuple with $M_j=0$ for $j\ne \ell$ and $M_\ell=M$. 
 Given $1\le \ell \le g$, define
 a bilinear form on $\bbF^n$ by
\[
  \cB_\ell(c,d)=\cL(cd^{\ast}\otimes \be_\ell)
\]
  for  $c,d\in\bbF^n$.
  There is a unique $n\times n$ matrix $B_\ell$  such that
$
  \cB_\ell(c,d)=\langle B_\ell c, d\rangle.
$

 Let $\Lambda_B$ denote the linear polynomial  $\Lambda_B(x)=\sum_1^g B_j x_j$.
 Fix a positive integer $m$ and let 
 $\{e_1,\dots,e_m\}$ denote the standard orthonormal basis for
  $\bbF^m$. 
 Let $Y=(Y_1,\dots,Y_g)\in\cC(m)$ be given and consider
  $\La_B(Y)$.
 Given vectors 
 $\gamma_p=\sum_{j=1}^m  \gamma_{p,j} \otimes e_j,$ for $p=1,2$,
 contained in $\bbF^n\otimes \bbF^m$, compute
\[
 \begin{split}
 \langle \Lambda_B(Y)\gamma_2,\gamma_1 \rangle
  & = \sum_{i,j}\sum_\ell \langle  B_\ell\gamma_{2,j},
           \gamma_{1,i}\rangle \langle Y_\ell e_j,e_i\rangle 
  =  \sum_{i,j} \sum_\ell \cL \big( \gamma_{2,j} \gamma_{1,i}^{\ast} \otimes \be_\ell \big)
                \langle Y_\ell e_j,e_i\rangle  \\
  & =  \cL \big( \sum_{\ell}  \sum_{i,j} \gamma_{2,i} \langle Y_\ell e_j,e_i\rangle \gamma_{1,j}^\ast \, 
            \otimes \be_\ell \big) 
  =   \cL \big( \sum_{\ell} \Gamma_2 Y_\ell \Gamma_1^{\ast}\otimes \be_\ell \big) 
   =  \cL(\Gamma_2  Y \Gamma_1^{\ast}),
 \end{split}
\]
 where  $\Gamma_p$ is the matrix with $j$-th column $\gamma_{p,j}$. 
 Using equation \eqref{eq:absLMI},
\[
 \begin{split}
 2 \operatorname{Re}(\cL(\Gamma_2 Y \Gamma_1^{\ast})) & \le  \tr(\Gamma_1^{\ast} T_1 \Gamma_1) +\tr(\Gamma_2^* T_2 \Gamma_2)  
  =  \sum_{p=1}^2 \sum_{j=1}^m \langle T_p \gamma_{p,j},\gamma_{p,j} \rangle  \\
   &= \sum_{p=1}^2  \langle (T_p\otimes I) \sum_j \gamma_{p,j}\otimes e_j,
       \sum_k \gamma_{p,k} \otimes e_k \rangle 
   = \sum_{p=1}^2 \langle (T_p\otimes I) \gamma_p,\gamma_p\rangle.
 \end{split}
\] 
 Thus, 
\begin{equation}
\label{eq:Tlmi}
 \Phi(Y) =\begin{pmatrix} T_1\otimes I & -\Lambda_B(Y) \\ -\Lambda_B(Y)^* & T_2\otimes I \end{pmatrix} \succeq 0
 \end{equation}
 for every $m$ and $Y\in\cC(m)$.

 Since $\cC$ contains the $\epsilon$-neighborhood of $0$,
  it contains $\pm \frac{\epsilon}{2} {\be_j}\in \bbF^g.$
  Hence, for each $j$,  
 \[
 \begin{split}
      0 & \preceq \Phi(\pm \frac{\epsilon}{2}{\be_j})  
      =  \begin{pmatrix}  T_1 & \pm \Lambda_B(\frac{\epsilon}{2}\be_j) \\ \pm \Lambda_B(\frac{\epsilon}{2} \be_j)^* & T_2 \end{pmatrix}   
      =  \begin{pmatrix}  T_1 & \pm \frac{\epsilon}{2} B_j \\ \pm \frac{\epsilon}{2} B_j^* &  T_2 \end{pmatrix}.
 \end{split}
 \] 
 Thus,  while the  $T_p$ need not be invertible, 
  it can be assumed (passing to subspaces of smaller dimension
  if necessary) that they are invertible. 
  Finally,
  multiplying left and right by $\oplus T_p^{-\frac12}$ produces the 
  linear polynomial  $\Lambda_Q(x) = \sum_j Q_j x_j$ (with $Q_j = T_1^{-\frac12} B_j T_2^{-\frac12}$)
  such that, with $\Psi$ denoting the monic symmetric linear pencil
\[
 \Psi(x) = \begin{pmatrix} I & -\Lambda_Q(x) \\ -\Lambda_Q(x)^* & I \end{pmatrix},
\]
 $\Psi(Y)\succeq 0$ if
  and only if $\Phi(Y) \succeq 0$. In particular, $\Psi$ is positive definite on $\cC$. Equivalently, $\|\Lambda_Q(Y)\|\le 1$ for all $Y\in \cC$.

 On the other hand,  computing as above, \eqref{eq:absLMI0}  becomes
\[
 \begin{split}
  \langle \Phi(\oX) e\oplus e,e\oplus e \rangle & =  \sum_{p=1}^2 \tr(T_p) - 2 \operatorname{Re}(\langle \Lambda_B(\oX)e,e\rangle ) \\
   &=  2-2 \sum_{\ell} \sum_{j,k} \operatorname{Re}( \langle B_\ell e_j,e_k \rangle \, \langle \oX_\ell e_j, e_k\rangle)\\
   & = 2 - 2 \sum_{\ell,j,k}\operatorname{Re}( \cL(e_je_k^*\otimes \be_\ell) \, \langle \oX_\ell e_j, e_k\rangle) \\
   &= 2 -2 \operatorname{Re} \Big(\cL [\sum_{\ell,j,k}  \langle \oX_\ell e_j, e_k\rangle\, (e_je_k^*\otimes \be_\ell)] \Big) = 2- 2 \operatorname{Re}(\cL(\oX)) =0,
 \end{split}
\]
 where  $e=\sum e_j\otimes e_j$.
 Since $\oX$ is in $\cC(n)$, it follows that
  $\Phi(\oX)\succeq 0$.
  Thus $\Phi(\oX)(e \oplus e)=0$,
  and since $(T_p \otimes I) e \ne 0,$ it follows
  that $\Psi(\oX)$ is singular too.  In particular, $\|\Lambda_Q(\oX)\|=1$.

 Finally, suppose $Y\in\cC$ and  $\|\Lambda_Q(Y)\|=1$. If $t>1$, then $\|\Lambda_Q(tY)\|>1$ and hence $tY\notin \cC$. Thus $Y$ is in the boundary of $\cC$. Hence  if $Y$ is in the interior of $\cC$, then $\|\Lambda_Q(Y)\|<1$.
\end{proof}

\begin{thm}[cf.~\protect{\cite[Proposition 3.5]{BVM}}]
\label{thm:OperatorBall}
 If $\cC \subseteq M(\bbF)^g$ is a closed matrix balanced, closed with respect to direct sums  and  $\cC$ contains a free $\epsilon$-neighborhood of $0,$ then 
$\cC$ is an operator pencil ball.
\end{thm}

\begin{lemma}
 \label{lem:Qbounded}
   If $\cC \subseteq M(\bbF)^g$ is a matrix convex set and if $\cC(1)$ contains $0$ in its interior, then there exists a constant $\kappa$ such that if $Q\in M(\bbF)^g$ and $\|\Lambda_Q(X)\|\le 1$ for all $X\in \cC$, then $\|Q_j\|\le \kappa$ for each $1\le j\le g$.
\end{lemma}

\begin{proof}
 Let $\{\be_j:1\le j\le g\}$ denote the standard basis for $\bbF^g$.  By hypothesis, there is an $\epsilon>0$ such that the tuple $\epsilon \be_j \in \cD_Q(1)$. Hence, $1\ge \|\Lambda(\epsilon \be_j)\| = \epsilon \|Q_j\|$. Choosing $\kappa =\frac{1}{\epsilon}$ completes the proof.
\end{proof}

\begin{proof}[Proof of Theorem~\ref{thm:OperatorBall}]
For a fixed $n$, choose a countable set $K(n)\subseteq \partial C(n)$ with $\overline{K(n)}=\partial C(n)$. By assumption $\cC$ contains a free $\epsilon$-neighborhood of $0$, so Proposition \ref{prop:sharp} implies that for each $X \in K(n)$ there exists a tuple $Q_X\in M_n(\bbF)^g$ such that 
 $\|\Lambda_{Q_X} (Y)\|\le 1$ for all $m$ and
 $Y\in\cC(m)$  and such that  $\|\Lambda_{Q_X} (X)\|=1$. 

Set $K=\bigcup_n K(n)$ and define $Q=\bigoplus_{X \in K} Q_X
$. Since $K(n)$ is countable for each $n$, it follows that $K$ is also countable. Furthermore, 
$Q$ is a bounded operator by Lemma \ref{lem:Qbounded}.
 We will show $\cC=\{X \in M(\bbF)^g: \|\Lambda_Q(X)|| \leq 1 \}$.

By construction, $\|\Lambda_{Q_X} (Y)\| \leq 1$ for all $Y \in \cC.$ 
Hence $\cC \subseteq \{X\in M(\bbF)^g: \|\Lambda_Q(X)\| \le 1\}$. Moreover,
if $X \in K,$ then $\|\Lambda_{Q_X} (X) \|=1$.  Since $K$ is dense in $\partial \cC$ and $\Lambda_Q$ is continuous,  $\| \Lambda_Q (X) \|=1$ for all $X \in \partial \cC$.

Finally, suppose $Y \notin \cC$. Since $\cC$ contains a free $\epsilon$-neighborhood of $0$ there exists some  $t \in (0,1)$ such that $tY \in \partial \cC$. It follows that $\|\Lambda_Q (tY)\|=1$ and hence $\|\Lambda_Q (Y) \|=\frac{1}{t} >1$. Thus $\cC \supseteq\{X \subseteq M(\bbF)^g:\| \Lambda_Q (X) \|\leq 1\}$ and therefore, $\cC$ is the  operator pencil ball $\{X: \| \Lambda_Q(X)\|\le 1\}$.
\end{proof}

\end{document}